\documentclass[reqno,11pt]{amsart}

\usepackage[linktocpage]{hyperref}
\usepackage{inputenc}
\usepackage{hyperref}
\usepackage{graphicx}
\usepackage{amscd}
\usepackage{slashed}
\usepackage{amssymb}
\usepackage{amsmath}
\usepackage{bbm}
\usepackage{esint}
\usepackage{eurosym}
\usepackage{mathtools}

\usepackage{longtable}
\usepackage{comment}
\usepackage{enumitem}

\usepackage{dsfont}

\usepackage[usenames,dvipsnames]{pstricks}
\usepackage{epsfig}

\usepackage[mathscr]{eucal}
\numberwithin{equation}{section}
\allowdisplaybreaks[1]

\title[Dirac Solutions in Kerr-Newman]{Solutions to the Dirac Equation in Kerr-Newman Geometries including the black-hole region}

\author[C.\ Krpoun]{Christoph Krpoun*}
\author[O.\ M\"uller]{Olaf M\"uller**}
\email{Christoph.krpoun@mathematik.ur.de, mullerol@mathematik.hu-berlin.de}

\address{*: Fakult\"at f\"ur Mathematik \\ Universit\"at Regensburg \\ Universitätsstr. 31 \\ D-93040 Regensburg \\ Germany}
\address{**: Fakult\"at f\"ur Mathematik \\ Humboldt Universit\"at zu Berlin \\ Unter den Linden 6 \\ D-10099 Berlin \\ Germany}

\newtheorem{Def}{Definition}[section]

\newtheorem{Thm}[Def]{Theorem}
\newtheorem{Prp}[Def]{Proposition}
\newtheorem{Lemma}[Def]{Lemma}

\newcommand{\beq}{\begin{equation}}
	\newcommand{\eeq}{\end{equation}}
\newcommand{\ce}{\centering}

\newcommand{\Proof}{\begin{proof}}
	\newcommand{\QED}{\end{proof} \noindent}

\newcommand{\la}{\langle}
\newcommand{\ra}{\rangle}
\newcommand{\bra}{\langle}
\newcommand{\ket}{\rangle}
\newcommand{\Sl}{\mathopen{\prec}}
\newcommand{\Sr}{\mathclose{\succ}}

\newcommand{\C}{\mathbb{C}}
\newcommand{\R}{\mathbb{R}}

\newcommand{\Z}{\mathbb{Z}}

\newcommand{\mm}{\mathcal{M}}
\newcommand{\nn}{\mathcal{N}}
\newcommand{\Gg}{\mathcal{G}}
\newcommand{\hh}{\mathcal{H}}

\newcommand{\s}{\sigma}

\newcommand{\Idmat}{\mathbbm{1}}
\newcommand{\dif}[1]{\text{d}#1}

\newcommand{\Ss}{\mathcal{S}}

\newcommand{\V}{\noindent}

\def\bes#1\ees{\begin{align*} #1 \end{align*}}
\def\be#1\ee{\begin{align} #1 \end{align}}
\def\bs{\begin{split}}
	\def\es{\end{split}}

\begin{document}
	
	\maketitle
	
	\begin{abstract}
		
		\noindent We investigate the Dirac equation in Kerr-Newman space-time, using horizon penetrating coordinates (Eddington-Finkelstein-Coordinates) and the Newman-Penrose formalism to separate the equation into radial and angular systems of ordinary differential equations, and deriving the asymptotics of the radial solutions at infinity and at the Cauchy horizon. 
	\end{abstract}


	\section{Introduction}
	\noindent Black holes are frequently studied models in mathematical physics that often push theories to their limits. 
	Moreover, after the first detection of a binary system by the joint LIGO/Virgo collaboration in 2015 displaying blatant similarities with the predictions 
	from black hole theory \cite{LIGO_Virgo}, they are even more in the focus of research as a quite accurate description of observed astrophysical phenomena. Both aspects urge a better understanding of the these objects.
	One step towards this goal is to derive solutions of relativistic quantum mechanical equations, such as the Dirac equation, in black-hole space-times. The Kerr-Newman family consisting of the axisymmetric stationary one-ended asymptotically flat Einstein-Maxwell solutions seems a good starting point for this analysis, which indeed has been performed already for some limit cases within this family [\cite{R_ken_2017}, \cite{FinsterAsympKerrNew} or \cite{Finster_2000}]. We extend these approaches to the case of a general Kerr-Newman space-time $\mm_{M,a,Q}$ and investigate the Dirac solutions in a region including the black hole, from the Cauchy horizon to future null infinity. 
	
	\bigskip
	
	In this analysis, we use the Eddington-Finkelstein coordinates, a generalization of the well known Kruskal coordinates, defined as follows: Starting from the usual Boyer-Lindquist coordinates $(t,r,\varphi, \theta) $ well-adapted to the symmetries of $\mm_{M,a,Q}$, but singular at the event horizon $r^{-1} (r_+)$ and the Cauchy horizon $r^{-1} (r_-)$, one defines a tortoise coordinate $r_{*}$ by
	
	\beq
	\ce
	r_{\star} := r + \dfrac{r_+^2 + a^2}{r_+ - r_-}\ln\big(|r-r_+|\big) - \dfrac{r_-^2 + a^2}{r_+ - r_-}\ln\big(|r-r_-|\big)  \in O_\infty (r) 
	\label{EFr}
	\eeq
	
	\V (with $r_* \rightarrow_{r \rightarrow r_\pm} {\mp} \infty$), a Cauchy temporal function $\tau := t + r_{\star} - r$ and an azimuthal coordinate $\phi := \varphi + \tilde{\varphi}$, forming the Eddington-Finkelstein coordinate system 
	$(\tau,r,\phi, \theta)$ in the range $\R \times (0, \infty) \times [0,2\pi) \times (0,\pi)$ over which the corresponding components $g^{EF}$ of $g$ are smooth.
	
	\bigskip
	
	\V After separating the Dirac equation in a radial and angular ODE (first done by Chandrasekhar for Kerr spacetime in \cite{ChandBH}) we obtain an asymptotic behaviour of the resulting radial ordinary differential equation in a Kerr-Newman space-time.

	\bigskip
	\bigskip

	\begin{Thm}
		The solutions of the radial ODE~\eqref{radialeODE} have the following asymptotics,
	\begin{enumerate}[leftmargin=2em]
	\item[{\rm{(i)}}] {\bf (Asymptotics at infinity)} Let $w_1 \in \C$ be the root of $\omega^2 - m^2$ contained in the convex hull of $\R_+ $ and $\R_+ \cdot i$ and $w_2 = -w_1$ 
        the other root, and let $\Theta := \frac{1}{4}\ln (\frac{\omega - m}{\omega + m} )$, then there is 
        $f_{\infty} := (f_{\infty}^{(1)}, f_{\infty}^{(2)})^T \in\R^2 \setminus \{ 0\} $ with
        \be
        \label{theo1}
        X_\omega(u) 
        &= \begin{bmatrix} 
            \cosh(\Theta) & \sinh(\Theta) \\
            \sinh(\Theta) & \cosh(\Theta) \\
        \end{bmatrix}	 \begin{pmatrix}
            f_{\infty}^{(1)}e^{i\Phi_+(u)} \\
            f_{\infty}^{(2)}e^{-i\Phi_-(u)} \\
        \end{pmatrix}
        + E_{\infty}(u)  \\
        \nonumber
        \ee
        \V for the asymptotic phases 
        \be
        \ce
        \Phi_{\pm}(u):= \pm w_{1,2}\, u + M \bigg( 2\omega \pm \dfrac{m^2}{w_{1,2}}\bigg) \ln(u)
        \label{asyPhase}
        \ee
        and for an error function $E_{\infty}(u)$ with polynomial decay. More precisely, there is $c \in \R_+$ with
        \bes
        \ce
        ||E_{\infty}|| \leqslant \dfrac{c}{u}.
        \ees
\item[{\rm{(ii)}}] {\bf (Asymptotics at the Cauchy horizon)}For every non-trivial solution $X$,
        \be
        \ce
        X_\omega(u) = 
        \begin{pmatrix}
            h_{r_-}^{(1)}e^{2i \omega u} \\
            h_{r_-}^{(2)}
        \end{pmatrix}
        + E_{r_-}(u)
        \label{theo:2}
        \ee
        with~$h_{r_-} := (h_{r_-}^{(1)}, h_{r_-}^{(2)})^T \in \R^2 \setminus \{0\} $, with $E_{r_-}$ such that for $r$ sufficiently close to $r_-$ and suitable constants $a,b \in \R_+$,
        \bes
        \ce
        ||E_{-}(u)|| \leq a e^{- b u}.
        \ees
\end{enumerate}
	\end{Thm}

\V Our result on the asymptotics towards Infinity matches the one obtained in Boyer-Lindquist coordinates in \cite{FinsterAsympKerrNew}.

	\noindent The paper has following structure: In Section 2, via the separation method, we show Eq. \ref{radialeODE}. 
	In Section 3 we perform the proof for the asymptotics. Three lengthy but straightforward calculation steps have been postponed to an appendix.
	
	\V In a subsequent work we will use these asymptotics to get an integral representation of the Dirac propagator, 
	to derive the fermionic signature operator, and ultimately, to find an alternative, observer-independent, definition for the Hawking radiation.
	
	\V The authors gratefully acknowledge substantial support of this work by Felix Finster. 
	
	\section{Deriving the integral representation}
	\subsection{Preliminaries on the tetrad formalism}
	\noindent We start with defining the necessary objects in the underlying theory. Let $(\mm,g)$ be a globally hyperbolic spacetime. We denote the spinor bundle of the manifold $\mm$ by $\s: \Ss \rightarrow \mm$, with fiber $\Ss_x$ over $x \in \mm$. 
	On each fiber we have a fibre metric $\Sl \cdot | \cdot \Sr_x$ with signature $(n,n)$ where $n = 2^{\lfloor k/2\rfloor -1}$ and $k = \text{dim}(\mm)$. The Clifford multiplication is defined via the maps of the general Dirac matrices $G$
	
	\beq
	\ce
	G : T_x \mm \longrightarrow L(\Ss_x) \quad \text{with} \quad G(u)G(v) + G(v)G(u) = 2 g(u,v)\Idmat_{\Ss_x}
	\label{diracMat}
	\eeq
	
	\noindent and is written out in components $G^{\mu}$. Throughout the article, we will use the Einstein convention $x^{\mu}y_{\mu} = \sum^n_{k=1} x^ky_k$.
	
	\V As a short-hand notation, we use the Feynmann dagger $\slashed{u} = u^{\mu}G_{\mu} = G(u)$. We denote the metric connection on $\s$ by $\nabla$. Furthermore, $C^{\infty}(\sigma)$ denotes the space of smooth 
	sections of $\sigma$. On $C^\infty (\s)$ we define canonically a globally, Lorentz invariant, but indefinite scalar product
	
	\beq
	\ce
	\begin{split}
		\bra \cdot | \cdot \ket : \, &C^{\infty}(\s) \times C^\infty_{sc}(\s) \longrightarrow \C \\
		&\bra \phi | \psi \ket = \int_M \prec \psi | \phi \succ_x \text{d}\mu_M 
	\end{split}
	\label{eq:1}
	\eeq
	\noindent We define the Dirac operator
	
	\beq
	\ce
	\Gg := iG^{\mu} \nabla_{\mu} + \mathcal{B} : C^{\infty}(M, \Ss M) \longrightarrow C^{\infty}(M, \Ss M), \nonumber
	\eeq
	
	\noindent with $\mathcal{B} \in L(\Ss_x)$ an external potential. This results in a Dirac equation for each constant mass parameter $m \in \R_+$
	
	\beq
	\label{DiracEq}
	\ce
	\big( \Gg -m \big)\psi_m = 0 .
	\vspace{2mm}
	\eeq
	
	\noindent The solutions of Eq. \ref{DiracEq} are denoted by $\psi_m$ and called {\em Dirac spinors} and $\hh_m$ the corresponding solution space. With the indefinite product in \eqref{eq:1} the tuple $(\hh_m, \bra \cdot | \cdot \ket)$ is a Krein space. For spectral calculations, it is useful that not on the space of spinors but on the space $\hh_m$ of Dirac spinors we have another natural, Lorentz invariant scalar product that is positive-definite (and thereby fixes a pre-Hilbert space structure), which is defined by 
	
	\beq
	\ce
	\big( \psi_m | \phi_m \big)_m = 2\pi \int_{\mathcal{N}} \prec \psi_m | \slashed{\nu} \phi_m \succ_x \text{d}\mu_{\mathcal{N}}(x)
	\vspace{2mm}
	\eeq
	
	\noindent where $\nn$ is a Cauchy hypersurface of $\mm$ with normal vector field $\nu$. A simple current conservation argument shows that this scalar product is independent of the choice of $\nn$.
	By completion we can extend $(\hh_m, (\cdot | \cdot))$ to a Hilbert space. Thus, we can derive a spin connection out of the general Dirac matrices $G^{\mu}(x)$ as derived in \cite{FinsterU22}. With this spin derivative it is possible to re-write the Dirac operator 
	in the alternative form
	\beq
	\ce
	\Gg = iG^{\mu}\mathcal{D}_{\mu}
	\label{diracOp}
	\eeq
	with the spin connection
	\be
	\ce
	\mathcal{D}_{\mu} &= \dfrac{\partial}{\partial x^{\mu}} - i\mathcal{E}_{\mu} - \mathcal{B}_{\mu} \nonumber \\
	\mathcal{E}_{\mu} &= \dfrac{i}{2}\rho(\partial_{\mu}\rho) - \dfrac{i}{16} \text{Tr}_{\Ss}(G^{\alpha}\nabla_{\mu}G^{\beta})G_{\alpha}G_{\beta} + \dfrac{i}{8} \text{Tr}_{\Ss}(\rho G_{\mu}\nabla_{\alpha}G^{\alpha})\rho
	\label{diracEq}
	\vspace{4mm}
	\ee
	where $\epsilon_{jklm}$ the Levi-Civita 
	symbol (totally antisymmetric tensor), $\rho = \frac{i}{4!}\epsilon_{jklm}G^jG^kG^lG^m$ and Tr$_{\Ss}$ denots the trace over the spinor indices.\\

	\noindent In mathematical relativity, it has proved advantageous for calculations to choose a basis of the tangent bundle independent of the local coordinates. This procedure is sometimes called the tetrad formalism. Tetrads are pseudo-orthonormal frames, i.e., consisting of one unit time-like and three unit space-like vectors all perpendicular to each other. 
	We will start by defining the main objects in our four-dimensional space-time. Choose a local frame $\{\partial_{\mu}\}$ with $\mu \in \{0,1,2,3\}$ of $TM$. Denote the dual basis by $\{\dif{x}^{\mu}\}$ of $T^{\ast}M$.\\
	
	\begin{Def}
		Let $(M,g)$ be a four-dimensional Lorentzian manifold. A tetrad is a smooth local section of the psdeudo-orthonormal frame bundle. The tetrad can be written in a local coordinate frame $(\partial_\mu)_{\mu = 1, ..., n}$ as
		
		\be
		\ce
		e_{(a)} = e_{(a)}^{\mu}\partial_{\mu} \quad \text{and} \quad e^{(a)} = e^{(a)}_{\mu}\dif{x}^{\mu}.
		\ee
	\end{Def}
	
	\noindent Any local coordinate frame (which always exists) is mapped by the Gram-Schmidt algorithm (which is smooth) to a tetrad, whose existence is thereby granted over every coordinate neighborhood.
	Because our underlying manifold is four-dimensional, one can picture the transformation maps $e_{(a)}^{\mu}$ as $4\times 4$ matrices. We will label the general space-time coordinates with Greek letters and the local 
	Lorentz space-time with Latin letters in brackets.\newline
	
	\noindent The Newmann-Penrose formalism is a variant of the tetrad formalism, where instead of requiring orthonormality of the ${e_{(a)}}$, one chooses 
	in the complexified tangent bundle $T_{\C}M = TM \otimes \C$ with complex conjugation $J: T_{\C}M \longrightarrow T_{\C}M$ denoted by $\bar{m} = J(m)$, a local null basis comprising of a pair of real null vectors $l$ and $n$ as well as a pair of complex conjugate null vectors $m$ and $\bar{m}$ in the orthogonal complement of $l$ and $n$.

	\begin{Def}
		\label{defNP}
		Let $(M,g)$ be a four-dimensional Lorentzian manifold, $TM$ the corresponding tangent bundle and $T_{\C}M$ its complexification. A {\bf Newman-Penrose tetrad} or {\bf double null frame} is a frame consisting at each point of 
		two real valued vectors $l, n \in T M$ and a pair of complex conjugated vectors $m, \bar{m} \in T_\C M$, with
		\be
		\ce
		g(l,m) = g(l,\bar{m}) = g(n,m) = g(n,\bar{m}) =& 0 \nonumber \\
		g(l,l) = g(n,n) = g(m,m) = g(\bar{m},\bar{m}) =& 0
		\ee
		and
		\bes
		\ce
		g(l,n) = 1 , \qquad g(m,\bar{m}) = -1 .
		\ees
	\end{Def}
	\noindent For later usage, we want to ensure the existence and smoothness of such a tetrad:
	
	\begin{Prp}
		Locally, a smooth Newman-Penrose tetrad exists. 
	\end{Prp}
	
	\begin{proof}
		Let $\{ \partial_{\mu}\}$ be a smooth local frame of $T_{\C}M$ where the elements of $\{ \partial_{\mu}\}$ do not here to be null vectors.
		Then we can apply Gram-Schmidt to those vectors and result in an smooth orthonormal basis $\{e_{(t)}, e_{(x)}, e_{(y)}, e_{(z)} \}$. Afterwards, we can define the double null frame via the orthonormal basis by the equations:
		\bes
		\ce
		l &= \dfrac{1}{\sqrt{2}}(e_{(t)} + e_{(z)}),  \qquad  n = \dfrac{1}{\sqrt{2}}(e_{(t)} - e_{(z)})\\
		m &= \dfrac{1}{\sqrt{2}}(e_{(x)} + i e_{(y)}), \qquad  \bar{m} = \dfrac{1}{\sqrt{2}}(e_{(x)} - i e_{(y)})
		\ees
		One easily checks that the above vectors fulfil all conditions from Definition \eqref{defNP}.
	\end{proof}
	
	\noindent We also single out a part of the corresponding gauge group that is isomorphic to $(\C\setminus \{ 0\}, \cdot)$ and that is later used to derive the so-called Cartan tetrad:

	\begin{Def}
		\label{lorenzTrafo}
		Let $(M,g)$ be a four-dimensional Lorentzian manifold. Let $C: M \rightarrow \C\setminus \{ 0 \}$. The {\bf rotation of class 3 corresponding to $C$} is the transformation
		
		\be
		\ce
		(l,n,m,\bar{m}) \mapsto (C \cdot l, C^{-1} \cdot n, \frac{C}{\vert C \vert} m, \frac{\bar{C}}{\vert C \vert} \bar{m}).
		\ee

	\end{Def}
	
	\subsection{Kerr-Newman spacetime and Eddington-Finkelstein-Coordinates}
	
	\noindent The Kerr-Newman metrics are the stationary axisymmetric solutions to the Einstein-Maxwell equations with one asymptotically flat end. Appropriate restrictions of the Kerr-Newman metrics form a $3$-dimensional manifold of Einstein-Maxwell solutions parametrized by the mass parameter $M$, the electric charge $Q$ and the angular momentum mass ratio $a = \frac{J}{M}$ [see \cite{jW}]. The underlying four-dimensional spacetime manifold is $ \mm = \R^2 \times S^2$ with coordinates chart $(t,r,\varphi, \theta)$. They are in the range $\R \times  (r_-; \infty) \times [0,2\pi) \times (0,\pi)$, for $r_- \in (0; \infty)$ defined below. We will use the signature $\text{sig}(+,-,-,-)$ and natural units. 
	The metric in Boyer-Lindquist coordinates reads (see \cite{ONeill}): 
	
	\be
	\ce
	g = &\bigg[ 1+ \dfrac{Q^2 - 2Mr}{\Sigma} \bigg] \dif{t}\otimes \dif{t}  - \sin ^2(\theta)\bigg[ r^2 + a^2 - \dfrac{a^2(Q^2 - 2Mr)}{\Sigma}\sin ^2(\theta) \bigg] \dif{\varphi}\otimes \dif{\varphi}  \nonumber \\
	& - \dfrac{\Sigma}{\Delta} \dif{r}\otimes \dif{r} - a\dfrac{\sin^2(\theta)}{\Sigma}\bigg[ Q^2 - 2Mr \bigg]\big(\dif{t}\otimes \dif{\varphi} + \dif{\varphi}\otimes \dif{t} \big) - \Sigma \dif{\theta}\otimes \dif{\theta} .
	\ee

	\beq
	\text{with} \quad \Delta(r) := r^2 - 2M r + a^2 + Q^2 \quad \text{and} \quad \Sigma(r, \theta) := r^2 + a^2\cos ^2(\theta). \nonumber 
	\vspace{2mm}
	\eeq
	
	\noindent $r_{\pm} := \pm \sqrt{M^2 - (a^2+Q^2)} +M$ are the two zero loci of $\Delta(r)$ in slow (i.e. $a^2+Q^2 < M^2$) Kerr-Newman spacetimes, they are positive.\newline
	\newline
	\noindent We have two coordinate singularities at $r_{\pm}$ which define two horizons, separating the subset \textbf{B$_{I}$} $ :=\{ x \in M \vert r(x) > r_+ \} $ outside of the black hole from the subset between event and Cauchy horizon
	\textbf{B$_{II}$} $:= \{ x \in M \vert r_- < r (x) < r_+ \} $. To get rid of the coordinate singularity at $r_+$, we introduce the Eddingtion-Finkelstein-Coordinates. 
	The tortoise coordinate $r_*$ is calculated from the tangent vectors of null geodesics. The latter have the form (see \cite{ChandBH} pp. 299 and 579)
	
	\bes
	\ce
	\dfrac{\dif{r}}{\dif{s}} = \dfrac{r^2 + a^2}{\Delta}E \qquad  \dfrac{\dif{r}}{\dif{s}} = \pm E \qquad \dfrac{\dif{\theta}}{\dif{s}} = 0 \qquad \dfrac{\dif{\varphi}}{\dif{s}} = \dfrac{a}{\Delta}E
	\vspace{2mm}
	\ees
	\noindent with E denoting the total energy and $s$ being an affine parameter. Consequently, by solving the equation $\frac{\dif{t}}{\dif{r}} = \pm \frac{r^2 + a^2}{\Delta}$ we can define the tortoise coordinates $r_{*}$
	
	\beq
	\ce
	r_{\star} := r + \dfrac{r_+^2 + a^2}{r_+ - r_-}\ln\big(|r-r_+|\big) - \dfrac{r_-^2 + a^2}{r_+ - r_-}\ln\big(|r-r_-|\big) 
	\label{EFr}
	\eeq
	\noindent As the space-time is for $a \neq 0$ not spherically symmetric but only axisymmetric, we have an additional tangent vector in  the $\phi$ direction. Thus, one has to transform the azimuthal angle
	
	\beq
	\ce
	\dfrac{\dif{\varphi}}{\dif{r}} = \pm \dfrac{a}{\Delta} \Longleftrightarrow \varphi = \pm \underbrace{\dfrac{a}{r_+ - r_-}\ln
		\bigg(\bigg| \dfrac{r-r_+}{r-r_-}\bigg|\bigg)}_{=:\tilde{\varphi}} + D	
	\label{EFphi}		
	\eeq
	with $D$ as a constant of integration. The sign corresponds to the distinction between ingoing and outgoing null geodesics. We will focus on the minus sign, which describes ingoing geodesics. We define a temporal function
	$\tau := t + r_{\star} - r$ and and an azimuthal coordinate $\phi := \varphi + \tilde{\varphi}$, forming the Eddington-Finkelstein coordinate system 
	$(\tau,r,\phi, \theta)$ in the range $\R \times (0,\infty) \times [0,2\pi) \times (0,\pi)$. 
	The metric expressed in these coordinates is
	
	\bes
	\ce
	g =& \bigg(1 + \dfrac{Q^2 - 2Mr}{\Sigma} \bigg) \dif{\tau}\otimes\dif{\tau} - \Sigma\dif{\theta} \otimes \dif{\theta} - \Sigma \sin^2(\theta) \dif{\phi} \otimes \dif{\phi} \\
	+ &\dfrac{Q^2-2Mr}{\Sigma}\bigg[\big(\dif{r} - a\sin^2(\theta)\dif{\phi}\big) \otimes \dif{\tau} + \dif{\tau} \otimes \big(\dif{r}-a\sin^2(\theta)\dif{\phi}\big)\bigg]  \\
	- & \bigg[ 1+ \dfrac{Q^2-2Mr}{\Sigma}\bigg]\big(\dif{r} - a \sin^2(\theta)\dif{\phi}\big) \otimes \big(\dif{r} - a \sin^2(\theta)\dif{\phi}\big)
	\vspace{2mm}
	\ees
	
	\noindent All hypersurfaces of constant $\tau$ are space-like:
	
	\begin{Lemma}[see \cite{R_ken_2017} for the Kerr case]
		$\tau$ is a temporal function on $B_{I} \cup B_{II}$.
		\vspace{2mm}
	\end{Lemma}
	
	\begin{proof}
		In the basis $(r, \phi, \theta)$, we want to show positive-definiteness for $ -g^{EF}\big|_{\tau = \text{const.}}$ of 
		
		\beq
		\ce
		A: = 
		\begin{pmatrix}
			+1 - \dfrac{Q^2-2Mr}{\Sigma} & -\bigg(1- \dfrac{Q^2-2Mr}{\Sigma}\bigg)a\sin ^2(\theta) & 0 \\
			-\bigg(1- \dfrac{Q^2-2Mr}{\Sigma}\bigg)a\sin ^2(\theta) & \sin ^2(\theta)\bigg[ r^2 + a^2 + \dfrac{a^2(2Mr-Q^2)}{\Sigma}\sin ^2(\theta)\bigg] & 0 \\
			0 & 0 & \Sigma
		\end{pmatrix} \nonumber.
		\vspace{2mm}
		\eeq
		Sylvester's criterion tells us that is enough to show positivity of the determinant of the diagonal minors of $A$. Starting with $\det_{11} (A)$ and knowing that $\Sigma > 0 \, \forall \, r$, we focus on $(\Sigma + 2Mr - Q^2)$. Inserting for $\Sigma$ and investigating the roots of the polynomial in $r$, we obtain the solutions $\rho_\pm = - M \pm \sqrt{M^2 + Q^2 - a^2\cos^2(\theta)}$. Moreover,
		
		\be
		\ce
		\rho_+ &= - M + \sqrt{M^2 + Q^2 - a^2\cos^2(\theta)} = - (M - \sqrt{M^2 + Q^2 - a^2\cos^2(\theta)}) \leq \nonumber \\
		&\leq M - \sqrt{M^2 + Q^2 - a^2\cos^2(\theta)} \leq M - \sqrt{M^2 - Q^2 - a^2} = r_- \nonumber
		\vspace{2mm}
		\ee
		This implies $\det_{11} (A)> 0 \, \forall\, r \geq r_-$. After some calculation we see that $\det_{22} (A) = \Sigma + 2Mr - Q^2$. By the same argument we
		have positivity $\forall \, r \geq r_-$. Since $\det_{33} (A) = \Sigma(\Sigma + 2Mr - Q^2) $ we have that $-g^{EF}\big|_{\tau = \text{const.}}$ is 
		positive definite $\forall \, r \geq r_-$. 
	\end{proof}
	
	\subsection{Computing the Cartan Tetrad in Eddingtion-Finkelstein-Coordinates}
	
	\noindent Now, we compute the Dirac operator in the Kerr-Newmann geometry. We start defining the Cartan tetrad which will be used to derive the operator in a matrix representation.
	We want to highlight that the general Dirac matrices are not uniquely defined by the anti-commutation relation \eqref{diracMat}. We can use this freedom to choose the matrices in a 
	way which simplifies the calculations. This can be done without any problems, since all possible sets of general Dirac matrices, satisfying 
	the anti-commutation relation \eqref{diracMat}, result in unitarily equivalent Dirac operators [see \cite{FinsterU22}].
	
	\noindent From a double null frame $\{l,n,m, \bar{m}\}$ we can define a real, orthonormal tetrad $u_{(a)}$. We will label the general space-time coordinates with Greek letters and the local Lorentz space-time with Latin letters in brackets:
	
	\be
	\ce
	u_{(0)} &= \dfrac{1}{\sqrt{2}}(l+n) \qquad u_{(1)} = \dfrac{1}{\sqrt{2}}(m + \bar{m})\nonumber \\
	u_{(2)} &= \dfrac{1}{\sqrt{2}i}(m- \bar{m})  \qquad    u_{(3)} = \dfrac{1}{\sqrt{2}}(l-n) 
	\label{ONBasis}
	\vspace{2mm}
	\ee
	
	\noindent Since we have an orthonormal frame the dyad metric is Minkowskian
	
	\beq
	\ce
	g_{\mu \nu}(x)u^{\mu}_{(a)}(x)u^{\mu}_{(b)}(x) = \eta_{(a)(b)} \qquad \eta^{(a)(b)}	u^{\mu}_{(a)}(x)u^{\mu}_{(b)}(x) = g^{\mu \nu}(x) 
	\vspace{2mm}
	\eeq
	
	\noindent with $\eta_{\mu\nu} = \eta^{\mu\nu} = \text{diag}(1,-1,-1,-1)$. Furthermore, we choose the gamma matrices to be in Weyl representation of the Minkowski space. They have the form
	
	\beq
	\ce
	\gamma^{(0)} = 
	\begin{bmatrix}
		0 & - \Idmat_{2 \times 2} \\
		- \Idmat_{2 \times 2} & 0 \\
	\end{bmatrix}
	\qquad 
	\gamma^{(i)} = 
	\begin{bmatrix}
		0 & -\sigma^{(i)} \\
		-\sigma^{(i)} & 0 \\
	\end{bmatrix} \nonumber
	\vspace{2mm}
	\eeq
	
	\noindent where $\sigma^{(i)}$ are the Pauli matrices and $i \in \{1,2,3\}$
	
	\beq
	\ce
	\sigma^{(1)} = 
	\begin{pmatrix}
		0 & 1 \\
		1 & 0\\
	\end{pmatrix}
	\qquad
	\sigma^{(2)} = 
	\begin{pmatrix}
		0 & -i \\
		i & 0\\
	\end{pmatrix}
	\qquad
	\sigma^{(3)} = 
	\begin{pmatrix}
		1 & 0 \\
		0 & -1\\
	\end{pmatrix} \nonumber
	\vspace{2mm}
	\eeq
	\noindent The $\gamma^{(a)}$ matrices satisfy the anti-commutation relation
	
	\beq
	\ce
	\eta^{(a)(b)} = \dfrac{1}{2} \big\{\gamma^{(a)}, \gamma^{(b)} \big\} . \nonumber
	\eeq
	
	\noindent We define generalized Dirac matrices, associated to the Kerr-Newmann metric:
	
	\beq
	\ce
	G^{\mu}(x) := u^{\mu}_{(a)}\gamma^{(a)} .
	\vspace{2mm}
	\eeq
	
	\V Indeed, the above defined matrices $G^{\mu}(x)$ satisfy the anti-commutation relations \eqref{diracMat}:
	
	\beq
	\ce
	\dfrac{1}{2}\bigg\{ G^{\mu}(x), G^{\nu}(x) \bigg\} = \dfrac{1}{2} u^{\mu}_{(a)}u^{\nu}_{(b)} \bigg\{\gamma^{(a)}, \gamma^{(b)} \bigg\} = u^{\mu}_{(a)}u^{\nu}_{(b)} \eta^{(a)(b)} = g^{\mu \nu}(x) .
	\eeq

	\V In this formulation, the spin inner product takes the form
	
	\be \label{spininner}
		\Sl \psi | \phi \Sr_x = \la \psi, \begin{pmatrix} 0 & \Idmat_{\C^2} \\ \Idmat_{\C^2} & 0 \end{pmatrix} \phi \ra_{\C^4} \:.
	\ee
	
	\noindent We now choose a certain frame for our computation. We start with a symmetric frame and apply a coordinate transformation into Eddington-Finkelstein coordinates. 
	In a last step we act with a rotation of class 3 (see Def. \ref{lorenzTrafo}) associated to a function $C = u \cdot e^{i \omega}$ for real functions $u, \omega$ to get rid of remaining singularities.
	
	\bes
	\ce
	l &= \dfrac{1}{\sqrt{2\Sigma|\Delta|}}\bigg( (r^2 + a^2)\partial_t + \Delta\partial_r + a \partial_{\varphi} \bigg) \\
	n &= \dfrac{\epsilon(\Delta)}{\sqrt{2\Sigma |\Delta|}}\bigg( (r^2 + a^2)\partial_t - \Delta\partial_r + a \partial_{\varphi} \bigg) \\
	m &= \dfrac{1}{\sqrt{2\Sigma}}\bigg( ia \sin(\theta)\partial_t + \partial_{\theta} + i \csc(\theta) \partial_{\varphi} \bigg) \\
	\bar{m} &= \dfrac{1}{\sqrt{2\Sigma}}\bigg( -ia \sin(\theta)\partial_t + \partial_{\theta} - i \csc(\theta) \partial_{\varphi} \bigg) .
	\ees
	
	\noindent Next, we transform the tetrad into its contra-variant form and afterwards apply the coordinate transformation from equation \eqref{EFr} and \eqref{EFphi} This results in
	
	\bes
	\ce
	l_{ED}^D &= \epsilon(\Delta)\sqrt{\dfrac{|\Delta|}{2\Sigma}}\bigg[ \dif{\tau} + \bigg(1-\dfrac{2\Sigma}	{\Delta}\bigg)\dif{r} -a\sin ^2(\theta)\dif{\phi} \bigg] \\
	n_{ED}^D &= \sqrt{\dfrac{|\Delta|}{2\Sigma}}\bigg[\dif{\tau} + \dif{r} -a\sin ^2(\theta)\dif{\phi}\bigg] \\
	m_{ED}^D &= \dfrac{1}{\sqrt{2\Sigma}}\bigg[ia\sin(\theta)\big(\dif{\tau} + \dif{r}\big) - \Sigma\dif{\theta} - i\big( r^2 + a^2 \big)\sin(\theta)\dif{\phi} \bigg]  \\
	\bar{m}_{ED}^D  &= -\dfrac{1}{\sqrt{2\Sigma}}\bigg[ia\sin(\theta)\big(\dif{\tau} + \dif{r}\big) + \Sigma\dif{\theta} - i\big( r^2 + a^2 \big)\sin(\theta)\dif{\phi} \bigg] .
	\ees
	
	\noindent Following, we transform the tetrad in its covariant form and apply another general Lorentz transformation of the form $u = \frac{\sqrt{|\Delta|}}{r_+}, \quad \omega = 0$. Finally, we end up with the 
	covariant Newman-Penrose tetrad in Eddington-Finkelstein-Coordinates for the Kerr-Newman space-time 
	
	\be
	\ce
	l' &= \dfrac{1}{\sqrt{2\Sigma}r_+}\bigg[ \big(2r^2 +2a^2 - \Delta \big)\partial_{\tau} + \Delta\partial_r + 2a\partial_{\phi} \bigg] \nonumber \\
	n' &= \dfrac{r_+}{\sqrt{2\Sigma}}\bigg[\partial_{\tau} - \partial_r \bigg] \nonumber \\
	m' &= \dfrac{1}{\sqrt{2\Sigma}}\bigg[ia\sin(\theta)\partial_{\tau} + \partial_{\theta} + i\csc(\theta)	\partial_{\phi} \bigg] \nonumber \\
	\bar{m}' &= -\dfrac{1}{\sqrt{2\Sigma}}\bigg[ia\sin(\theta)\partial_{\tau} - \partial_{\theta} + i\csc(\theta)\partial_{\phi} \bigg]
	\label{tetradCov}
	\ee
	
	\noindent and its contra variant form
	
	\be
	\ce
	l'_D &= \dfrac{1}{\sqrt{2\Sigma}r_+}\bigg[ \Delta\dif{\tau} + \big(\Delta -2\Sigma \big)\dif{r} -a\Delta\sin ^2(\theta)	\dif{\phi} \bigg] \nonumber \\
	n'_D &= \dfrac{r_+}{\sqrt{2\Sigma}}\bigg[\dif{\tau} + \dif{r} -a\sin ^2(\theta)\dif{\phi}\bigg] \nonumber \\
	m'_D &= \dfrac{1}{\sqrt{2\Sigma}}\bigg[ia\sin(\theta)\big(\dif{\tau} + \dif{r}\big) - \Sigma \dif{\theta} - i\big( r^2 + a^2 \big)\sin(\theta)\dif{\phi} \bigg] \nonumber \\
	\bar{m}'_D  &= -\dfrac{1}{\sqrt{2\Sigma}}\bigg[ia\sin(\theta)\big(\dif{\tau} + \dif{r}\big) + \Sigma\dif{\theta} - i\big( r^2 + a^2 \big)\sin(\theta)\dif{\phi} \bigg] .
	\label{tetradContr}
	\vspace{2mm}
	\ee
	
	\subsection{Computing the Dirac Equation}
	
	\noindent We plug in the the orthonormal frames \eqref{tetradContr} and \eqref{tetradCov} into \eqref{diracEq} and compute the Dirac operator. 
	Due to our choice the computations simplifies a lot because $\rho \equiv \gamma^5 = i \gamma^0\gamma^1\gamma^2\gamma^3$ is a constant factor which reduces
	$\mathcal{E}_{\mu}$ [see \cite{Finster_2000}]:
	
	\beq
	\ce
	\mathcal{E}_{\mu} = -\dfrac{i}{16}\text{Tr}_{\Ss}(G^{\alpha}\nabla_{\mu}G^{\beta})G_{\alpha}G_{\beta} .
	\vspace{2mm}
	\eeq
	
	\noindent With the the commutation relation \eqref{diracMat} and its consequence $4\nabla_{\alpha}g^{\mu \nu} = \text{Tr}_{\Ss}[(\nabla_{\alpha}G^{\mu})G^{\nu}] + \text{Tr}_{\Ss}[G^{\mu}(\nabla_{\alpha}G^{\nu})]$, we can re-write the trace term as in \cite{Finster_2000}:
	
	\beq
	\ce
	\text{Tr}_{\Ss}(G^{\alpha}\nabla_{\mu}G^{\beta})G^{\mu}G_{\alpha}G_{\beta} = -8\nabla_{\mu}G^{\mu} + i\epsilon^{\mu \alpha \beta \delta}
	\text{Tr}_{\Ss}(G_{\alpha}\nabla_{\mu}G_{\beta})\gamma^5G_{\delta} . \nonumber
	\vspace{2mm}
	\eeq
	Given the Levi-Civita connection is torsion free, one can replace the covariant derivative by a partial derivative. In the end we have
	
	\beq
	\ce
	B := G^{\mu}E_{\mu} = \dfrac{i}{2\sqrt{|g|}}\partial_{\mu}\big(\sqrt{|g|}u^{\mu}_{(a)})\gamma^{(a)} 
	-\dfrac{1}{4}\epsilon^{\mu \alpha \beta \delta}\eta^{(a)(b)}u_{(a)\alpha}\big(\partial_{\mu}u_{(b)\beta}\big)u_{(c)\delta}\gamma^{(c)}\gamma^5 .
	\label{BTerm}
	\vspace{2mm}
	\eeq
	
	\noindent In the Appendix (Sec. \ref{erste}) we compute
	
	\be
	\label{ErsteGleichung}
	\ce
	B &= \dfrac{i(r-M)}{2 \sqrt{\Sigma}r_+}\bigg[\gamma^{0} + \gamma^{(3)}\bigg] +\dfrac{i\cot(\theta)}{2\sqrt{{\Sigma}}} \gamma^{(1)} + \dfrac{ia^2}{2\Sigma^{3/2}}\cos(\theta)\sin(\theta)\gamma^{(1)}  \\ \nonumber
	&+ \dfrac{ir}{4\Sigma^{3/2}r_+}\bigg[ (\Delta - r_+^2)\gamma^{(0)} + (\Delta + r_+^2)\gamma^{(3)}\bigg] \\ \nonumber
	&+ \dfrac{a(\Delta - r_+^2)}{4\Sigma^{3/2}r_+}\cos(\theta)\gamma^{(0)}\gamma^5 + \dfrac{a(\Delta + r_+^2)}{4\Sigma^{3/2}r_+}\cos(\theta)\gamma^{(3)}\gamma^5 \\ \nonumber
	&+ \dfrac{r a \sin(\theta)}{2\Sigma^{3/2}}\gamma^{(1)}\gamma^5 .
	\vspace{2mm}
	\ee
	
	\noindent Additionally, we need to compute a second term containing the partial derivatives
	\bes
	\ce
	iG^{\mu}\dfrac{\partial}{\partial x^{\mu}} &= i\gamma^{(0)}\bigg[u^0_{(0)}\partial_{\tau} + u^1_{(0)}\partial_{r} + u^3_{(0)}\partial_{\phi}\bigg]  \\
	&+ i\gamma^{(1)}u^2_{(1)}\partial_{\theta}  + i\gamma^{(2)}\bigg[u^0_{(2)}\partial_{\tau} + u^3_{(2)}\partial_{\phi}\bigg]  \\
	&+ i\gamma^{(3)}\bigg[u^0_{(3)}\partial_{\tau} + u^1_{(3)}\partial_{r} + u^3_{(3)}\partial_{\phi}\bigg] .
	\vspace{2mm}
	\ees
	
	\noindent With the definition in \eqref{diracOp} we show in the Appendix (\ref{zweite}) that the Dirac Operator has the following matrix representation
	\beq
	\label{ZweiteGleichung}
	\ce
	G = \begin{bmatrix}
		0 & 0 & \alpha_1 & \beta_- \\
		0 & 0 & \beta_+ & \alpha_0& \\
		\bar{\alpha}_0 &\bar{\beta}_+  & 0& 0\\
		\bar{\beta}_- & \bar{\alpha}_1 & 0& 0\\
	\end{bmatrix}, \qquad \text{with}
	\vspace{2mm}
	\eeq
	
	\be
	\ce
	\alpha_1 :=& -\dfrac{i}{\sqrt{\Sigma}r_+}\bigg[ \big(  2 r^2 + 2a^2 - \Delta\big)\partial_{\tau} + 2a\partial_{\phi} \bigg] \nonumber \\
	&-\dfrac{\Delta}{\sqrt{\Sigma}r_+}\bigg[ i \partial_r + i\dfrac{r-M}{\Delta} + \dfrac{i}{2\Sigma}\big( r - ia	\cos(\theta)\big) \bigg] \nonumber \\
	\alpha_0 :=& -\dfrac{ir_+}{\sqrt{\Sigma}}\bigg[\partial_{\tau} - \partial_r \bigg] + \dfrac{ir_+}{2\Sigma^{3/2}}\bigg[r - ia\cos(\theta)\bigg]\nonumber \\
	\bar{\alpha}_1 :=& -\dfrac{i}{\sqrt{\Sigma}r_+}\bigg[ \big( 2 r^2 + 2a^2 - \Delta \big)\partial_{\tau} + 2a\partial_{\phi} 	\bigg] 	\nonumber \\
	&-\dfrac{\Delta}{\sqrt{\Sigma}r_+}\bigg[ i \partial_r + i\dfrac{r-M}{\Delta} + \dfrac{i}{2\Sigma}\big( r + ia	\cos(\theta)\big) \bigg] \nonumber \\
	\bar{\alpha}_0 :=& -\dfrac{ir_+}{\sqrt{\Sigma}}\bigg[\partial_{\tau} - \partial_r \bigg] + -\dfrac{ir_+}{2\Sigma^{3/2}}\bigg[r + ia\cos(\theta)\bigg] \nonumber \\
	\beta_- :=& -\dfrac{1}{\sqrt{\Sigma}}\bigg[i \partial_{\theta} + i \dfrac{\cot(\theta)}{2} + \dfrac{a\sin(\theta)}{2\Sigma}	\big(r - ia\sin(\theta)\big)\bigg] \nonumber \\
	&- \dfrac{1}{\sqrt{\Sigma}}\big(a\sin(\theta)\partial_{\tau} + \csc(\theta)\partial_{\phi} \big) \nonumber \\
	\beta_+ :=& -\dfrac{1}{\sqrt{\Sigma}}\bigg[i \partial_{\theta} + i \dfrac{\cot(\theta)}{2} + \dfrac{a\sin(\theta)}{2\Sigma}	\big(r - ia\sin(\theta)\big)\bigg] \nonumber \\
	&+ \dfrac{1}{\sqrt{\Sigma}}\big(a\sin(\theta)\partial_{\tau} + \csc(\theta)\partial_{\phi} \big) \nonumber \\
	\bar{\beta}_- :=& \dfrac{1}{\sqrt{\Sigma}}\bigg[i \partial_{\theta} + i \dfrac{\cot(\theta)}{2} - \dfrac{a\sin(\theta)}{2\Sigma}\big(r + ia\sin(\theta)\big)\bigg] \nonumber \\
	&- \dfrac{1}{\sqrt{\Sigma}}\big(a\sin(\theta)\partial_{\tau} + \csc(\theta)\partial_{\phi} \big) \nonumber \\
	\bar{\beta}_+ :=& \dfrac{1}{\sqrt{\Sigma}}\bigg[i \partial_{\theta} + i \dfrac{\cot(\theta)}{2} - \dfrac{a\sin(\theta)}{2\Sigma}\big(r + ia\sin(\theta)\big)\bigg] \nonumber \\
	&+ \dfrac{1}{\sqrt{\Sigma}}\big(a\sin(\theta)\partial_{\tau} + \csc(\theta)\partial_{\phi} \big) .
	\vspace{2mm}
	\ee
	
	\noindent We can transform the Dirac equation into a simpler form. To this aim, we define the transformation matrix $D := \text{diag}(\bar{\delta}^{1/2}, (\bar{\delta}|\Delta|)^{1/2}, (\delta|\Delta|)^{1/2}, \delta^{1/2})$ 
	with $\delta := (r+ia\cos(\theta))$ and $\bar{\delta}:= (r-ia\cos(\theta))$. Furthermore set $\Gamma := -i\text{diag}(\delta, - \delta, -\bar{\delta}, \bar{\delta})$ and $\hat{\Psi} := D\Psi$, then we get
	\beq
	\label{DritteGleichung}
	\Gamma D \big(G-m\big) D^{-1} \hat{\Psi} = 
	\begin{bmatrix}
		i\delta m & 0 & \sqrt{|\Delta|}^{-1}\tilde{\mathcal{D}_1} &\tilde{\mathcal{L}_+} \\
		0 & -i\delta m & -\tilde{\mathcal{L}_-} &\sqrt{|\Delta|}\tilde{\mathcal{D}_0}\\
		\sqrt{|\Delta|}\tilde{\mathcal{D}_0} & \tilde{\mathcal{L}_+} & -i\bar{\delta}m & 0 \\
		-\tilde{\mathcal{L}_-} & \sqrt{|\Delta|}^{-1}\tilde{\mathcal{D}_1} & 0 & i\bar{\delta}m \\
	\end{bmatrix}\hat{\Psi}, 
	\eeq
	\noindent with the differential operators
	\be
	\ce
	\tilde{\mathcal{D}_1} &= \dfrac{1}{r_+} \bigg[ \big( 2 r^2 + 2a^2 - \Delta \big) \partial_{\tau} + \Delta\partial_r + 2a\partial_{\phi}	\bigg] \nonumber \\
	\tilde{\mathcal{D}_0} &=-r_+\bigg[\partial_{\tau} - \partial_r \bigg] \nonumber \\
	\tilde{\mathcal{L}_{\pm}} &=  \partial_{\theta} + \dfrac{\cot(\theta)}{2}  \mp i \big(a\sin(\theta)\partial_{\tau} + \csc(\theta)\partial_{\phi} \big) .
	\ee
	
	\V This is shown in the Appendix, Sec. \ref{dritte}.
	
	\subsection{Separation of the Dirac Equation}
	
	As in \cite{ChandBH} for the equation in Kerr spacetime, we can separate Eq. \ref{DritteGleichung} into angular and radial equations and thus can employ the eigenvalue problems 
	
	\bes
	\ce
	\mathcal{R}(r)\hat{\Psi} = \xi \hat{\Psi} \quad \text{and} \quad \mathcal{A}(\theta)\hat{\Psi} = - \xi \hat{\Psi}
	\ees
	with the solution
	\be
	\ce
	\hat{\Psi}(\tau,r,\theta,\phi) = e^{-i\omega\tau}e^{-ik\phi}&\hat{\Phi}(r,\theta) \quad \text{with:} \quad \omega \in \R, k \in (\Z +1/2)\nonumber \\
	\hat{\Phi}(r,\theta) &= \begin{pmatrix}
		\tilde{X}(r)_2Y(\theta)_2 \\
		\tilde{X}(r)_1Y(\theta)_1 \\
		\tilde{X}(r)_1Y(\theta)_2 \\
		\tilde{X}(r)_2Y(\theta)_1 \\
	\end{pmatrix},
	\vspace{2mm}
	\ee
	where $\xi$ is the constant of separation. The matrices $\mathcal{R}$ and $\mathcal{A}$ are defined as
	
	\be
	\ce
	\mathcal{R}(r) = 
	\begin{bmatrix}
		imr & 0 & \sqrt{|\Delta|}^{-1}\mathcal{D}_1 & 0 \\
		0 & -imr & 0 & \sqrt{|\Delta|}\mathcal{D}_0 \\
		\sqrt{|\Delta|}\mathcal{D}_0 & 0 & -imr & 0 \\
		0 & \sqrt{|\Delta|}^{-1}\mathcal{D}_1 & 0 & imr\\
	\end{bmatrix}
	\ee
	and
	\begin{spreadlines}{5mm}
		\be
		\ce
		\mathcal{A}(\theta) = 
		\begin{bmatrix}
			-am\cos(\theta) & 0 & 0 & \mathcal{L}_+ \\
			0& am\cos(\theta)& -\mathcal{L}_- & 0 \\
			0 & \mathcal{L}_+ & -am\cos(\theta) & 0 \\
			-\mathcal{L}_- & 0 & 0 & am\cos(\theta)\\
		\end{bmatrix} .
		\ee		
		The differential operators have the form
	\end{spreadlines}
	\be
	\ce
	\mathcal{D}_1 &= -\dfrac{1}{r_+} \bigg[ \big( 2 r^2 + 2a^2 - \Delta \big) i\omega - \Delta\partial_r + 2aki\bigg] \nonumber \\
	\mathcal{D}_0 &= r_+\bigg[i\omega + \partial_r \bigg] \nonumber \\
	\mathcal{L_{\pm}} &=  \partial_{\theta} + \dfrac{\cot(\theta)}{2}  \mp  \big(a\omega\sin(\theta)+ k\csc(\theta) \big) .
	\vspace{2mm}
	\ee	 
	In the end, the equations decouple into two systems of ordinary differential equations
	
	\begin{spreadlines}{10mm}
		\be
		\label{XGleichung}
		\ce
		\begin{bmatrix}
			\sqrt{|\Delta|}^{-1} \mathcal{D}_1 & imr - \xi \\
			-(imr +\xi) & \sqrt{|\Delta|}\mathcal{D}_0 \\
		\end{bmatrix}
		\begin{pmatrix}
			\tilde{X}_1(r)\\
			\tilde{X}_2(r)\\
		\end{pmatrix}
		&= 0,
		\ee
		
		\be	
		\label{YGleichung}	
		\ce
		\begin{bmatrix}
			\mathcal{L}_+ &  -am\cos(\theta) + \xi \\
			am\cos(\theta) + \xi & -\mathcal{L}_- \\
		\end{bmatrix}
		\begin{pmatrix}
			Y_1(\theta)\\
			Y_2(\theta)\\
		\end{pmatrix}
		&= 0 .
		\ee
	\end{spreadlines}

	\noindent By applying another partial derivative in $\theta$ to the angular ordinary differential equation \ref{YGleichung} we obtain the second-order Chandrasekhar-Page equation \cite{ChandBH}. Its solutions are usually referred to as spin-$\frac{1}{2}$ 
	spheroidal harmonics. In the limit $a \searrow 0$, these solutions reduce to the spin-weighted spherical harmonics for the spin-$\frac{1}{2}$ case \cite{GoldSphe}. For future work, we were only interested in that the matrix valued differential
	operator $\mathcal{A}$ has a spectrum decomposition with smooth eigenfunctions and discrete non-degenerate eigenvalues. Define $Y(\theta) := (Y_1(\theta), Y_2(\theta))$.
	
	\begin{Prp}
		For any $\omega \in \R$ and $k \in (\Z + 1/2)$ the differential operator $\mathcal{A}$ has a complete set of orthonormal eigenfunctions $(\hat{Y}_n)_{n \in \Z} \subset L^2((0,\pi), \sin(\theta)d\theta)$ that are bounded and smooth away from the poles.
		The corresponding eigenvalues $\xi_n$ are real-valued, discrete and non-degenerated. Both the eigenfunctions and the eigenvalues depend smoothly on $\omega$.
	\end{Prp}
	
	\begin{proof}
		This was already proven in \cite{FinsterAsympKerrNew}. This ends the proof.
	\end{proof}
	
	\noindent Now let us focus on the radial ordinary differential equation $\mathcal{R}(r)$. We will slightly re-write the ordinary differential equation in a more symmetric form. This is accomplished 
	by defining $X_2(r):= r_+\tilde{X}_2(r)$, $X_1 (r) := \tilde{X}_1(r)$ and $X(r):=(X_1(r),X_2(r))^T$:
	
	\bes
	\ce
	\partial_r X(r) = \underbrace{\dfrac{1}{\Delta}
		\begin{bmatrix}
			i[\omega( 2 r^2 + 2a^2 - \Delta) +2ka] & \sqrt{|\Delta|}(-imr + \xi) \\
			\epsilon(\Delta)\sqrt{|\Delta|}(imr + \xi) & -i\Delta \omega
	\end{bmatrix}}_{=:\tilde{U}(r)} X(r)
	\ees
	One can see, that the above matrix has singularities of rank zero at $r= r_{\pm}$. We can get rid of these singularities when we express the radial ordinary differential equation in tortoise coordinates introduce in \eqref{EFr}.
	
	\beq
	\ce
	\partial_{r\star} X = U(r_{\star}) X  \qquad \text{for} \qquad U(r_{\star}) :=  \bigg(\dfrac{\Delta}{r^2+ a^2}\tilde{U}(r)\bigg)(r_{\star}) 
	\label{radialeODE}
	\eeq
	\newpage
	
	\section{Analysis of the Radial Ordinary Differential Equation}

	\subsection{Analysis at Infinity}
	
	\begin{proof}
		Since the matrix potential \eqref{radialeODE} converges for $r_{\star} \longrightarrow \infty$, we define its limit
		\be
		\ce
		\
		U_{\infty} := \lim\limits_{r\star \rightarrow \infty}{U(r_{\star})} = i
		\begin{pmatrix}
			\omega & -m \\
			m & -\omega \\
		\end{pmatrix} .
		\label{eq:23}
		\ee
		This matrix has the eigenvalues
		\bes
		\ce
		\lambda^{(0)}= \pm
		\begin{cases}
			-\sqrt{m^2 - \omega^2} = w_1 \in \R \qquad &m > |\omega| \\
			i\sqrt{\omega^2 - m^2} = w_2 \in \C \qquad &m < |\omega|
		\end{cases} .
		\ees
		Thus, the eigenvalues for the case $m > |\omega|$ are real. This leads to non-trivial solutions which grow exponentially with $\sim e^{w_1 r_{\star}} $ and decay exponentially with $\sim e^{-w_1 r_{\star}}$, 
		which is a consequence of \eqref{eq:43}.\\
		
		\noindent In a first step, we want to diagonalize the matrix potential \eqref{radialeODE} with the diagonalization matrix $D(r_{\star})$ and $\lambda_{1,2}$ being the eigenvalues of $U(r_{\star})$. We call this matrix 
		$S := D^{-1}U(r_{\star})D = \text{diag}(\lambda_1, \lambda_2)$. Since $U(r_{\star})$ has a regular expansion in a power series of $\frac{1}{r_{\star}}$ and as the spectrum depends smoothly on the matrix, this also holds for the eigenvalues $\lambda_{1,2}$ and $D$.\newline
		Thus, we get $D(r_{\star}) \in D_{\infty} + \mathcal{O}(\frac{1}{r_{\star}})$ where $D_{\infty}$ is the diagonalization matrix for \eqref{eq:23}. The terms of order $\mathcal{O}(\frac{1}{r_{\star}})$ can be 
		absorbed in $E_{\infty}(r_{\star})$ and by direct computation one can show the form in \eqref{theo1}.\newline
		Now, we re-write \eqref{radialeODE} in the form
		\bes
		\ce
		\partial_{r\star} \big(D^{-1}X\big) = \bigg[ S - D^{-1}\big(\partial_{r\star}D\big)\bigg] \big(D^{-1}	X\big) .
		\ees
		Therefore, it is reasonable to impose the ansatz
		\be
		\ce
		X(r_{\star}) := D(r_{\star}) 
		\begin{pmatrix}
			e^{i\Phi_+(r_{\star})}f^{(1)}(r_{\star}) \\
			e^{-i\Phi_-(r_{\star})}f^{(2)}(r_{\star}) \\
		\end{pmatrix} .
		\label{eq:41}
		\ee
		for real functions $\Phi_\pm$. We put $f := (f^{(1)},f^{(2)})^T$. Note that Eq. \ref{radialeODE} implies that $X$, if nontrivial, is nonzero for large $r_\star$, and together with Eq. \ref{eq:41} we conclude that $f$ is nonzero for large $r_\star$ for any nontrivial solution. We obtain the ODE 
		\bes
		\ce
		\partial_{r\star} f = \bigg[ S - W^{-1}D^{-1}\big(\partial_{r\star}D\big) W - W^{-1}\big(\partial_{r\star} 	W\big)\bigg]f
		\ees
		with $W (r_\star)= \text{diag}(e^{i\Phi_+(r_{\star})},e^{-i\Phi_-(r_{\star})})$. Additionally, choosing $\Phi_\pm$ such that
		$S = W^{-1} \big(\partial_{r\star} W 	\big)$ leads to the ODEs
		\be
		\ce
		\partial_{r\star}\Phi_+(r_{\star}) &= -i \lambda_1 , \qquad \partial_{r\star}\Phi_-(r_{\star}) = i 	\lambda_2   \\
		\partial_{r\star} f &= - W^{-1}D^{-1}\big(\partial_{r\star}D\big) W f .
		\label{eq:43}
		\ee
		To calculate the asymptotic phases, we will look at the first order expansion in $\frac{1}{r_\star}$ of the matrix potential
		\be
		\ce
		U(r_{\star}) \in i 
		\begin{pmatrix}
			\omega & m \\
			-m & -\omega \\
		\end{pmatrix}
		+ \dfrac{1}{r_{\star}} 
		\begin{pmatrix}
			2\omega M & -\xi + imM \\
			-\xi - imM & 2\omega M \\
		\end{pmatrix} + O(\frac{1}{r_\star^2}).
		\label{eq:44}
		\ee
		This leads to a first order expansion in the eigenvalues\footnote{The eigenvalues have terms of order $1/r_{\star}^2$. Therefore, it is necessary to expand those in a power series of $1/r_{\star}$ as well.}
		\be
		\ce
		\lambda_{1,2}^{(1)} \in  \pm i w_{1,2} + \dfrac{iM}{r_{\star}}	\bigg(2\omega \pm \dfrac{m^2}{w_{1,2}}\bigg)  + \mathcal{O}\bigg(\dfrac{1}{r_{\star}^2}\bigg) .
		\ee
		Now, one can plug in these eigenvalues in \eqref{eq:43} and ends up with the required asymptotic phases in \eqref{asyPhase}.\footnote{Due to the logarithmic behaviour the $1/r_{\star}$ terms in the eigenvalues are essential.}\newline
		It remains to show that $f$ has a non-trivial, finite limit and that the error term is of polynomial decay. We will start with the limit. Since $D(r_{\star})$ has a regular $\frac{1}{r_{\star}}$ expansion, for $r_{\star}$ sufficiently close to infinity we can bound the Hilbert-Schmidt norms
		\be
		\ce
		||D^{-1} (r_\star)||_{HS} \leq c , \qquad ||\partial_{r\star} D (r_\star)||_{HS} \leq \dfrac{d}{r_{\star}^2} \quad \text{with} \quad c,d > 0 .
		\ee
		Since $||W||_{HS} = ||W^{-1}||_{HS} = \sqrt{2} $ one can estimate the $\C^2$-norm of \eqref{eq:43}
		\be
		\ce
		||\partial_{r\star}f|| \leq 2 ||D^{-1}||_{HS} \cdot ||\partial_{r\star} D||_{HS} \cdot ||f|| \leq \dfrac{2cd}{r_{\star}^2} .
		\label{eq:27}
		\ee
		Keeping in mind that after Eq. \ref{eq:41} we saw that there is $D>0$ such that $\vert \vert f  \vert \vert  \neq 0$ for $r_\star >D$, with the Cauchy-Schwarz and triangle inequality we derive
		\be
		\ce
		\big|\partial_{r\star}||f|| \big| = \dfrac{|\partial_{r\star} \langle f,f \rangle|}{2 ||f||} \leq \dfrac{|\langle f,\partial_{r\star}f \rangle| + | \langle \partial_{r\star}f,f \rangle|}{2||f||} \leq \dfrac{||f|| \cdot ||\partial_{r\star}f||}{||f||} = ||\partial_{r\star}f|| .
		\ee
		Using \eqref{eq:27} we end up with the inequality
		\be
		\ce
		\big|\partial_{r\star}||f|| \big| \leq \dfrac{2cd}{r_{\star}^2}||f||
		\label{eq:28}
		\ee
		for any non-trivial solution. We conclude
		\be
		\ce
		&\bigg| \int_{r_0}^{r_\star} \partial_{r_\star '} \ln(||f||)\dif{r_{\star}'}\bigg| \leq 2cd \int_{r_0}^{r_\star}
		\dfrac{1}{(r_{\star}')^2}\dif{r_{\star}'} \nonumber \\
		\Longrightarrow &\big|\ln(||f||)\big| \bigg|_{r_0}^{r_\star} \leq - \dfrac{2cd}{r_{\star}'}\bigg|_{r_0}^{r_\star} .
		\ee
		Because of $0 < \frac{2cd}{r_{\star}'} \big|_{r_0}^{r_\star} < \infty$, there exists a constant $N > 0$ such that $\frac{1}{N} \leq ||f|| \leq N$. This implies, for sufficiently large $r_{\star}$,
		\be
		\ce
		||\partial_{r\star}f|| \leq \dfrac{b}{r_{\star}^2}
		\ee
		with $b = \frac{2cd}{N}$. Thus $r_\star \mapsto \partial_{r_\star}f$ is integrable and has a finite limit $f_{\infty}:= 
		\lim\limits_{r_\star \rightarrow \infty}{f(r_{\star})} \neq 0$. As a last step, we estimate the error $E_{\infty}$. Firstly,
		\be
		\ce
		||E_f|| = ||f-f_{\infty} || = \bigg|\bigg| \int_{r\star}^{\infty} \partial_{r\star '}f \dif{r_{\star} '} 			\bigg|\bigg| \leq \dfrac{b}{r_{\star}} .
		\label{eq:29}
		\ee
		The polynomial decay of the total error $E_{\infty}$ comes from the fact that the matrices $U$, $D$ and the eigenvalues $\lambda$ have a regular expansion in $\frac{1}{r_{\star}}$. When looking at the $\C^2$-norm of $E_{\infty}$ one obtains
		\be
		\ce
		||E_{\infty}|| &= || X - X_{\infty}|| = \bigg| \bigg| D
		\begin{pmatrix}
			f^{(1)}e^{i\Phi_+(r_{\star})} \\
			f^{(2)}e^{-i\Phi_-(r_{\star})} \\
		\end{pmatrix}
		-D_{\infty}
		\begin{pmatrix}
			f_{\infty}^{(1)}e^{i\Phi_+(r_{\star})} \\
			f_{\infty}^{(2)}e^{-i\Phi_-(r_{\star})} \\
		\end{pmatrix}
		\bigg| \bigg| \nonumber \\
		&= \bigg|\bigg|D_{\infty}W_{\infty}(f-f_{\infty}) + \mathcal{O}\big(\dfrac{1}{r_{\star}}\big) \bigg|\bigg| \leq 
		||D_{\infty}||_{HS}\cdot||W_{\infty}||_{HS}\cdot||E_f|| + \big|\big|\mathcal{O}\big(\dfrac{1}{r_{\star}}\big)\big|\big| .
		\ee
		Thus, with \eqref{eq:29} we can conclude that
		\be
		\ce
		||E_{\infty}|| \leq \dfrac{a}{r_{\star}} \qquad \text{with} \qquad a = \sqrt{2}b||D_{\infty}||_{HS} \in \R_+ .
		\ee
		This ends the proof.
	\end{proof}
	
	\subsection{Analysis at the Cauchy horizon}
	
	\begin{proof}
		We now prove the asymptotics described in Theorem \eqref{theo:2} for\footnote{The proof for the case $r \searrow r_+$ would be identical modulo a the sign flip in the exponential decay caused by the different asymptotic behaviour of $r_*$ at the horizons. However, the analysis at $r = r_+$ would not yield anything new for the case of smooth initial values given on a Cauchy surface for Boyer-Lindquist blocks I and II.} the case $r \searrow r_-$. \newline
		 Plugging the form \eqref{theo:2} of the solution into Eq. \eqref{radialeODE}, a short calculation yields
		
		\be
		\ce
		\partial_{r\star} h = \dfrac{i}{r^2+a^2}	
		\underbrace{\begin{bmatrix}
			-\omega\Delta - 2k\big(\Omega_{KN}^{(-)}\big(r^2+a^2) - a\big) & -e^{-2i\big(\omega + k\Omega^{(-)}_{KN}\big) r_{\star}}\sqrt{\Delta}\big(mr +i \xi \big) \\
			e^{2i\big(\omega + k\Omega^{(-)}_{KN}\big) r_{\star}}\sqrt{\Delta}\big(mr -i \xi \big) & -\omega\Delta \\
		\end{bmatrix}}_{=:B}h .
		\label{eq:2.16}
		\ee
		
		\noindent Note: $r \geq r_-$ implies $\epsilon(\Delta) = 1$. Since $r_-$ is a root of $\Delta$, the right-hand side of the above system vanishes in the limit $r \searrow r_-$. Therefore, the solution has a limit
		$h_{r_-} := \lim\limits_{r \rightarrow r_-}{h} \in \R^2 \setminus \{0\}$. \newline
		Next, we show that this is not a trivial solution. By Eq. \eqref{EFr}, for any $\epsilon \in (0; r_+ -r_-) $ we find constants $0< c_1 <c_2$ with
		\bes
		\ce
		c_1 e^{-2\alpha r_{\star}} + r_- < r < c_2 e^{-2\alpha r_{\star}} + r_-
		\ees
		for each $r \in (r_-, r_- + \epsilon)$ and for $\alpha := \frac{1}{2}\frac{r_+ - r_-}{r_-^2 + a^2} \in \R_+$.
	 Firstly, we note that this implies
		
			\bes
			\ce
			\Delta , \ \  \Omega_{KN}^{(-)}(r^2 + a^2) -a  \ \in \  \mathcal{O}(e^{-2 \alpha r_\star}) 
			\ees
			as functions of $r_\star$. Consequently, the operator norm $\vert B \vert $ of $B$ is in $\mathcal{O} (e^{- \alpha r_\star})$, and thus
		
		\bes
		\ce
		||\partial_{r\star}h|| \leq q{ e^{-\alpha r_{\star}}}||h||
		\ees
		with a suitable constant $q\in \R_+$. We repeat the steps of the previous proof and obtain
		\bes
		\ce
		\bigg| \ln(||h||)\big|^{r_0}_{r\star} \bigg| \leq \dfrac{q}{\alpha} e^{-\alpha r_{\star}'}\big|^{r_0}_{r\star}
		\ees
		$\forall r_{\star} \leq r_0$. Since the right-hand side is positive, there is $N  \in (0;\infty)$ with $\frac{1}{N} \leq ||h|| \leq N$. Combining both results yields
		\bes
		\ce
		||\partial_{r\star} h|| \leq b e^{-\alpha r_{\star}}
		\ees
		with $ b = qN$. Altogether, this implies that $r_\star \mapsto \partial_{r_\star} h$ is integrable and has a finite, non-zero limit for $r_{\star} \rightarrow \infty$.\\
		Finally, we estimate the error as a function of $r_{\star}$. We start with the error in $h$
		\bes
		\ce
		||E_h|| := ||h - h_{r_-}|| = \bigg|\bigg| \int^{r_{\star}}_{-\infty} \partial_{r\star '} h \dif{r_{\star '}} \bigg|\bigg| \leq be^{-\alpha r_{\star}}.
		\ees
		Thus the total error $E_{r_-}$ has the form
		
		\be
		\ce
		||E_{r_-}|| = ||X - X_{r_-}|| \leq ||h - h_{r_-}|| = ||E_h|| \leq b e^{-\alpha r_{\star}} .
		\ee
		This shows the exponential decay of $E_{r_-}$ and concludes the proof.
	\end{proof}

	\newpage
	\appendix

	\section{Proof of Eq. \ref{ErsteGleichung}}
	\label{erste}
	
	\noindent Writing the orthonormal basis \eqref{ONBasis} with the results from \eqref{tetradCov} and \eqref{tetradContr}, we get
	
	\be
	\ce
	u_{(0)}^{\mu} &= \dfrac{1}{2\sqrt{\Sigma}r_+}\bigg[ \big(2 r^2 + 2a^2 - \Delta +r_+^2\big)\partial_{\tau} + \big( \Delta - r_+^2 \big)\partial_{r} +2a\partial_{\phi} \bigg] \nonumber \\
	u_{(1)}^{\mu} &= \dfrac{1}{\sqrt{\Sigma}}\partial_{\theta} \nonumber \\
	u_{(2)}^{\mu} &= \dfrac{1}{\sqrt{\Sigma}}\bigg[ a\sin(\theta)\partial_{\tau} + \csc(\theta)\partial_{\phi}\bigg] \nonumber \\
	u_{(3)}^{\mu} &= \dfrac{1}{2\sqrt{\Sigma}r_+}\bigg[ \big(2 r^2 + 2a^2 - \Delta -r_+^2\big)\partial_{\tau} + \big( \Delta + r_+^2 \big)\partial_{r} +2a\partial_{\phi} \bigg]
	\label{eqA:2.1}
	\ee
	
	\V for the vectors, whereas for the forms we obtain
	
	\be
	\ce
	u_{(0)\mu} &= \dfrac{1}{2\sqrt{\Sigma}r_+}\bigg[\big( \Delta + r_+^2 \big)\dif{\tau} + \big(\Delta -2\Sigma + r_+^2\big)						\dif{r} - a\sin ^2(\theta)\big( \Delta + r_+^2\big)\dif{\phi}\bigg] \nonumber \\
	u_{(1)\mu} &= - \sqrt{\Sigma}\dif{\theta} \nonumber \\
	u_{(2)\mu} &= \dfrac{1}{\sqrt{\Sigma}}\bigg[ a\sin(\theta)\big(\dif{\tau} + \dif{r}\big) - \big(r^2 + a^2\big)\sin(\theta)								\dif{\phi} \bigg] \nonumber \\
	u_{(3)\mu} &= \dfrac{1}{2\sqrt{\Sigma}r_+}\bigg[\big( \Delta - r_+^2 \big)\dif{\tau} + \big(\Delta -2\Sigma - r_+^2\big)									\dif{r} - a\sin ^2(\theta)\big( \Delta - r_+^2\big)\dif{\phi}\bigg] .
	\ee
	We will start computing the first term in \eqref{BTerm}:
	\be
	\ce
	\dfrac{i}{2\sqrt{|g|}}\partial_j\big(\sqrt{|g|}u^j_{(a)}\big)\gamma^{(a)} = \underbrace{\dfrac{i}{2\sqrt{|g|}}\partial_j						\big(\sqrt{|g|}\big)u^j_{(a)}\gamma^{(a)}}_{=:I} + \underbrace{\dfrac{i}{2}\partial_j u^j_{(a)}\gamma^{(a)}}_{=:II}
	\ee
	\be
	\ce
	I &= \underbrace{\dfrac{i}{2\sqrt{|g|}}}_{=:B}\bigg[\big(\partial_r \sqrt{|g|}u^1_{(a)}\gamma^{(a)} + \big(\partial_{\theta}
	\sqrt{|g|}\big) \underbrace{u^2_{(a)}}_{a = 1}\gamma^{(a)} \bigg] \nonumber \\
	&=B \bigg[2r\sin(\theta)\big(u^1_{(0)}\gamma^{(0)} + u^1_{(3)}\gamma^{(3)} \big) + \Sigma \cos(\theta) u^2_{(1)}\gamma^{(1)}					\bigg] \nonumber \\
	&= \dfrac{i}{2}\bigg\{ \dfrac{r}{\Sigma^{3/2}r_+}\big[ (\Delta - r_+^2)\gamma^{(0)} + (\Delta + r_+^2)\gamma^{(3)}\big] + 					\dfrac{\cot(\theta)}{\sqrt{{\Sigma}}} \gamma^{(1)}\bigg	\} \nonumber \\
	&= \dfrac{ir}{2\Sigma^{3/2}r_+}\bigg[(\Delta - r_+^2)\gamma^{(0)} + (\Delta + r_+^2)\gamma^{(3)} \bigg] + 									\dfrac{i\cot(\theta)}{2\sqrt{{\Sigma}}} \gamma^{(1)}
	\ee
	\be
	\ce
	II &= \dfrac{i}{2}\bigg( \partial_r u^1_{(a)}\gamma^{(a)} + \partial_{\theta}u^2_{(a)}\gamma^{(a)}\bigg) \nonumber \\
	&= \dfrac{i}{2}\bigg( \partial_r u^1_{(0)}\gamma^{(0)} + \partial_r u^1_{(3)}\gamma^{(3)} + \partial_{\theta}u^2_{(1)}							\gamma^{(1)}\bigg) \nonumber \\
	&= \dfrac{i}{2r_+}\bigg[ \partial_r \bigg(\dfrac{\Delta - r_+ ^2}{\Sigma}\bigg)\gamma^{(0)} + \partial_r \bigg(\dfrac{\Delta + 					r_+ ^2}{\Sigma}\bigg)\gamma^{(3)}\bigg] + \dfrac{i}{2}\partial_{\theta}\dfrac{1}{\sqrt{\Sigma}}\gamma^{(1)} \nonumber \\
	&= \dfrac{i}{4r_+}\bigg[ \dfrac{2(r-M)-(\Delta - r_+ ^2)r}{\Sigma^{3/2}}\gamma^{(0)} + \dfrac{2(r-M)-(\Delta + r_+ ^2)r}						{\Sigma^{3/2}}\gamma^{(3)}\bigg] + \dfrac{ia^2}{\sqrt{2\Sigma^{3/2}}}\cos(\theta)\sin(\theta)\gamma^{(1)} \nonumber \\
	&= \dfrac{i(r-M)}{2 \sqrt{\Sigma}r_+}\bigg[\gamma^{0} + \gamma^{(3)}\bigg] - \dfrac{ir}{4\Sigma^{3/2}r_+}\bigg[ (\Delta - r_+^2)				\gamma^{(0)} + (\Delta + r_+^2)\gamma^{(3)}\bigg] + \dfrac{ia^2}{2\Sigma^{3/2}}\cos(\theta)\sin(\theta)\gamma^{(1)}
	\ee
	Adding both terms we get the first intermediate result
	\be
	\ce
	\label{FirstTerm}
	I + II &= \dfrac{i(r-M)}{2 \sqrt{\Sigma}r_+}\bigg[\gamma^{0} + \gamma^{(3)}\bigg] +i																		\dfrac{i\cot(\theta)}{2\sqrt{{\Sigma}}} \gamma^{(1)} + \dfrac{ia^2}{2\Sigma^{3/2}}\cos(\theta)\sin(\theta)									\gamma^{(1)} \nonumber \\
	&+ \dfrac{ir}{\Sigma^{3/2}r_+}\bigg[ \dfrac{1}{2}(\Delta - r_+^2)\gamma^{(0)} + \dfrac{1}{2}(\Delta + r_+^2)\gamma^{(3)} -						\dfrac{1}{4}(\Delta - r_+^2)\gamma^{(0)} - \dfrac{1}{4}(\Delta + r_+^2)\gamma^{(3)}\bigg] \nonumber \\
	&= \dfrac{i(r-M)}{2 \sqrt{\Sigma}r_+}\bigg[\gamma^{0} + \gamma^{(3)}\bigg] +\dfrac{i\cot(\theta)}{2\sqrt{{\Sigma}}} \gamma^{(1)} + 				\dfrac{ia^2}{2\Sigma^{3/2}}\cos(\theta)\sin(\theta)	\gamma^{(1)} \nonumber \\
	&+ \dfrac{ir}{4\Sigma^{3/2}r_+}\bigg[ (\Delta - r_+^2)\gamma^{(0)} + (\Delta + r_+^2)\gamma^{(3)}\bigg] .
	\ee
	The second term of \eqref{BTerm} is harder to compute. Firstly, we notice that $u^{\mu}_{(a)}$ is only a function of $r$ and $\theta$. Therefore $\mu \in \{1,2\}$. We will start with $\mu=1$ and look at the terms for each $(b)\in \{0,1,2,3\}$ independently.
	Here $\epsilon$ is the Levi-Civita symbol of curved space-time (the normed totally antisymmetric symbol) $\epsilon^{\mu \alpha \beta \delta} = \frac{\tilde{\epsilon}^{\mu \alpha \beta \delta}}{\sqrt{|g|}}$, so we can do the calculation with $\tilde{\epsilon}$, the Levi-Civita symbol of euclidean space-time, and divide the results by the factor $\sqrt{|g|}$.\newline
	\underline{$\mu = 1$ and $(b)=0$}:
	\be
	\ce
	&-\dfrac{1}{4}\tilde{\epsilon}^{1 \alpha \beta \delta}\underbrace{\eta^{(0)(0)}}_{=1}\underbrace{u_{(0)\alpha}}_{\alpha \neq 2}				\big(\partial_r \underbrace{u_{(0)\beta}}_{\beta \neq 2}\big)u_{(c)\delta}\gamma^{(c)}\gamma^5 \nonumber \\
	&=- \dfrac{1}{4}\tilde{\epsilon}^{1 \alpha \beta 2}u_{(0)\alpha}\big(\partial_r u_{(0)\beta}\big)u_{(1)2}\gamma^{(1)}						\gamma^5 \nonumber \\
	&=- \dfrac{1}{4} u_{(1)2}\gamma^{(1)}\gamma^5 \bigg[\underbrace{\tilde{\epsilon}^{1 0 3 2}}_{=1}u_{(0)0}\big(\partial_r u_{(0)3}				\big) + \underbrace{\tilde{\epsilon}}_{=-1}^{1 3 0 2}u_{(0)3}\big(\partial_r u_{(0)0}\big)\bigg] \nonumber \\
	&=- \dfrac{a\sin ^2(\theta)}{16 r_+^2 \sqrt{\Sigma}} u_{(1)2}\gamma^{(1)}\gamma^5\bigg[ -\big(\Delta + r_p^2\big) \partial_r 				\bigg( \dfrac{\Delta + r_p^2}{\sqrt{\Sigma}}\bigg) + \big(\Delta + r_p^2\big) \partial_r \bigg( \dfrac{\Delta + 							r_p^2}{\sqrt{\Sigma}}\bigg) \bigg] \nonumber \\
	&= 0
	\ee
	For the next cases we shorten the computations because they always include the same steps.\newline
	\underline{$\mu = 1$ and $(b)=1$}:
	\be
	\ce
	&-\dfrac{1}{4}\tilde{\epsilon}^{1 \alpha \beta \delta}\underbrace{\eta^{(1)(1)}}_{=-1}\underbrace{u_{(1)\alpha}}_{\alpha = 2}				\big(\partial_r \underbrace{u_{(1)\beta}}_{\beta = 2}\big)u_{(c)\delta}\gamma^{(c)}\gamma^5 = 0
	\ee
	\underline{$\mu = 1$ and $(b)=2$}:
	\be
	\ce
	... &= \dfrac{1}{4} u_{(1)2}\gamma^{(1)}\gamma^5 \bigg[u_{(2)0}\big(\partial_r u_{(2)3}\big) + 													u_{(2)3}\big(\partial_r u_{(2)0}\big)\bigg] \nonumber \\
	&=  u_{(1)2}\gamma^{(1)}\gamma^5 \dfrac{a \sin ^2(\theta)}{4\sqrt{\Sigma}}\bigg[ -\partial_r \bigg(\dfrac{r^2 + a^2}						{\sqrt{\Sigma}}\bigg) + (r^2+a^2)\partial_r\dfrac{1}{\sqrt{\Sigma}}\bigg] \nonumber \\
	&= u_{(1)2}\gamma^{(1)}\gamma^5 \dfrac{a \sin ^2(\theta)}{4\sqrt{\Sigma}}\bigg[\dfrac{-2r}{\sqrt{\Sigma}}\bigg] = 
	\dfrac{ra \sin ^2(\theta)}{2\sqrt{\Sigma}}\gamma^{(1)}\gamma^5
	\ee
	
	\underline{$\mu = 1$ and $(b)=3$}:
	\be
	\ce
	... &= \dfrac{1}{4} u_{(1)2}\gamma^{(1)}\gamma^5 \bigg[u_{(3)0}\big(\partial_r u_{(3)3}\big) + 													u_{(3)3}\big(\partial_r u_{(3)0}\big)\bigg] \nonumber \\
	&= \dfrac{a\sin ^2(\theta)}{16 r_+^2 \sqrt{\Sigma}} u_{(1)2}\gamma^{(1)}\gamma^5\bigg[ -\big(\Delta - r_p^2\big) \partial_r 				\bigg( \dfrac{\Delta - r_p^2}{\sqrt{\Sigma}}\bigg) + \big(\Delta - r_p^2\big) \partial_r \bigg( \dfrac{\Delta - 							r_p^2}{\sqrt{\Sigma}}\bigg) \bigg] \nonumber \\
	&= 0
	\ee
	This completes all possible terms for $\mu = 1$. The next step is doing the same for $\mu=2$. Again, starting with $(b)=0$.\newline
	\underline{$\mu = 2$ and $(b)=0$}:
	\be
	\ce
	&-\dfrac{1}{4}\tilde{\epsilon}^{2 \alpha \beta \delta}\underbrace{\eta^{(0)(0)}}_{=1}u_{(0)\alpha}\big(\partial_{\theta} u_{(0)				\beta}\big)u_{(c)\delta}\gamma^{(c)}\gamma^5 
	\ee
	This time we do not have sharp conditions on $\alpha$ and $\beta$. Therefore, we have to sum up over all permutations of $\{0,1,3\}$:
	\be
	\ce
	-\dfrac{1}{4}\bigg[ &\tilde{\epsilon}^{2 0 1 3}u_{(0)0}\big(\partial_{\theta} u_{(0)1}\big)u_{(c)3}\gamma^{(c)}\gamma^5 + 
	\tilde{\epsilon}^{2 1 0 3 }u_{(0)1}\big(\partial_{\theta} u_{(0)0}\big)u_{(c)3}	\gamma^{(c)}\gamma^5  \nonumber \\
	+ &\tilde{\epsilon}^{2 0 3 1}u_{(0)0}\big(\partial_{\theta} u_{(0)3}\big)u_{(c)1}\gamma^{(c)}\gamma^5 + 
	\tilde{\epsilon}^{2 3 1 0}u_{(0)3}\big(\partial_{\theta} u_{(0)1}\big)u_{(c)0}\gamma^{(c)}\gamma^5  	\nonumber \\	
	+ &\tilde{\epsilon}^{2 3 0 1}u_{(0)3}\big(\partial_{\theta} u_{(0)0}\big)u_{(c)1}\gamma^{(c)}\gamma^5 +
	\tilde{\epsilon}^{2 1 3 0}u_{(0)1}\big(\partial_{\theta} u_{(0)3}\big)u_{(c)0}\gamma^{(c)}\gamma^5\bigg]
	\ee
	This can be rearranged
	\be
	\ce
	&-\dfrac{1}{4}u_{(c)3}\gamma^{(c)}\gamma^5 \overbrace{\bigg[ \underbrace{\tilde{\epsilon}^{2 0 1 3}}_{=1}u_{(0)0}							\big(\partial_{\theta} u_{(0)1}\big) + \underbrace{\tilde{\epsilon}^{2 1 0 3 }}_{-1}u_{(0)1}\big(\partial_{\theta} 
		u_{(0)0}\big)\bigg]}^{=:I} \nonumber \\
	&-\dfrac{1}{4}u_{(c)1}\gamma^{(c)}\gamma^5 \overbrace{\bigg[ \underbrace{\tilde{\epsilon}^{2 3 0 1}}_{=1}u_{(0)3}							\big(\partial_{\theta} u_{(0)0}\big) + \underbrace{\tilde{\epsilon}^{2 0 3 1 }}_{-1}u_{(0)0}\big(\partial_{\theta} 
		u_{(0)3}\big)\bigg]}^{=:II} \nonumber \\
	&-\dfrac{1}{4}u_{(c)0}\gamma^{(c)}\gamma^5 \overbrace{\bigg[ \underbrace{\tilde{\epsilon}^{2 1 3 0}}_{=1}u_{(0)1}							\big(\partial_{\theta} u_{(0)3}\big) + \underbrace{\tilde{\epsilon}^{2 3 1 0 }}_{-1}u_{(0)3}\big(\partial_{\theta} 
		u_{(0)1}\big)\bigg]}^{=:III} .
	\ee
	Calculating the terms in the parenthesis first and then summing up over the last free index $(c)$:
	\be
	\ce
	I &= \dfrac{1}{4\sqrt{\Sigma}r_+^2}\bigg[ \big(\Delta + r_+^2)\partial_{\theta}\bigg(\dfrac{\Delta - 2\Sigma + r_+^2}							{\sqrt{\Sigma}}\bigg) - \big(\Delta - 2\Sigma + r_+^2\big)\partial_{\theta}\bigg(\dfrac{\Delta + r_+^2}{\sqrt{\Sigma}}						\bigg)\bigg] \nonumber \\
	&= \dfrac{a^2}{\Sigma r_+^2}\big(\Delta + r_+^2\big)\cos(\theta)\sin(\theta)
	\ee
	\be
	\ce
	II &= \dfrac{1}{4\sqrt{\Sigma}r_+^2}\big(\Delta + r_+^2\big)^2 a\bigg[ \partial_{\theta}\bigg( \dfrac{\sin ^2(\theta)}								{\sqrt{\Sigma}}\bigg) - \sin ^2(\theta)\partial_{\theta}\bigg(\dfrac{1}{\sqrt{\Sigma}}\bigg)\bigg] \nonumber \\
	&= \dfrac{(\Delta + r_+^2)^2 a}{\Sigma r_+^2}\sin(\theta)\cos(\theta)
	\ee
	\be
	\ce
	III &= \dfrac{a(\Delta + r_+^2)}{4\sqrt{\Sigma}r_+^2}\bigg[ \sin ^2(\theta)\partial_{\theta} \bigg(\dfrac{\Delta - 2 \Sigma + r_				+^2}{\sqrt{\Sigma}} \bigg) - \big( \Delta - 2 \Sigma + r_+^2 \big) \partial_{\theta} \bigg(\dfrac{\sin ^2(\theta)}							{\sqrt{\Sigma}} \bigg) \bigg] \nonumber \\
	&= \dfrac{a(\Delta + r_+^2)}{2\sqrt{\Sigma}r_+^2} \big( 2a^2 + 2r^2 - r_+^2 - \Delta \big) \cos(\theta) \sin(\theta)
	\ee
	Now, summing over $(c)$. Note that the $\gamma^{(1)}\gamma^5$-term is trivially zero.\newline
	\underline{$\gamma^{(0)}\gamma^5$-Term}:
	\be
	\ce
	-\dfrac{1}{4}\gamma^{(0)}\gamma^5\bigg[u_{(0)3}I + u_{(0)1}II + u_{(0)0}III\bigg] = 0
	\ee
	\underline{$\gamma^{(2)}\gamma^5$-Term}:		
	\be
	\ce
	-\dfrac{1}{4}\gamma^{(2)}\gamma^5\bigg[u_{(2)3}I + u_{(2)1}II + u_{(2)0}III\bigg] = 0
	\ee	
	\underline{$\gamma^{(3)}\gamma^5$-Term}:		
	\be
	\ce
	-\dfrac{1}{4}\gamma^{(3)}\gamma^5\bigg[u_{(3)3}I + u_{(3)1}II + u_{(3)0}III\bigg] = \dfrac{a(\Delta + r_+^2)}								{4\sqrt{\Sigma}r_+}\cos(\theta)\sin(\theta)\gamma^{(3)}\gamma^5
	\ee	
	\underline{$\mu = 2$ and $(b)=1$}:
	\be
	\ce
	-\dfrac{1}{4}\tilde{\epsilon}^{2 \alpha \beta \delta}\overbrace{\eta^{(1)(1)}}^{=-1}\underbrace{u_{(1)\alpha}}_{(= 0 ,\alpha 				\neq 2) }\big(\partial_{\theta}\underbrace{u_{(1)}\beta}_{( = 0, \beta \neq 2)}\big)u_{(c)\delta}\gamma^{(c)}\gamma^5 = 0
	\ee
	Therefore, all terms with $\mu = 2$ and $(b)=1$ are zero. The next step involves calculating for $(b)=2$. Again, we have to sum over all permutations of $\{0,1,3\}$ in the Levi-Civita symbol. Skipping the first parts we end up \newline
	\underline{$\mu = 2$ and $(b)=2$}:
	\be
	\ce
	... &=\dfrac{1}{4}u_{(c)3}\gamma^{(c)}\gamma^5 \overbrace{\bigg[ \underbrace{\tilde{\epsilon}^{2 0 1 3}}_{=1}u_{(2)0}							\big(\partial_{\theta} u_{(2)1}\big) + \underbrace{\tilde{\epsilon}^{2 1 0 3 }}_{-1}u_{(2)1}\big(\partial_{\theta} 
		u_{(2)0}\big)\bigg]}^{=:I} \nonumber \\
	&+\dfrac{1}{4}u_{(c)1}\gamma^{(c)}\gamma^5 \overbrace{\bigg[ \underbrace{\tilde{\epsilon}^{2 3 0 1}}_{=1}u_{(2)3}							\big(\partial_{\theta} u_{(2)0}\big) + \underbrace{\tilde{\epsilon}^{2 0 3 1 }}_{-1}u_{(2)0}\big(\partial_{\theta} 
		u_{(1)3}\big)\bigg]}^{=:II} \nonumber \\
	&+\dfrac{1}{4}u_{(c)0}\gamma^{(c)}\gamma^5 \overbrace{\bigg[ \underbrace{\tilde{\epsilon}^{2 1 3 0}}_{=1}u_{(2)1}							\big(\partial_{\theta} u_{(2)3}\big) + \underbrace{\tilde{\epsilon}^{2 3 1 0 }}_{-1}u_{(2)3}\big(\partial_{\theta} 
		u_{(2)1}\big)\bigg]}^{=:III}
	\ee
	\be
	\ce
	I = \dfrac{a^2}{\sqrt{\Sigma}}\bigg[ - \partial_{\theta}\bigg(\dfrac{\sin ^2(\theta)}{\sqrt{\Sigma}}\bigg) + 
	\partial_{\theta}\bigg(\dfrac{\sin ^2(\theta)}{\sqrt{\Sigma}}\bigg)\bigg] = 0
	\ee
	\be
	\ce
	II = \dfrac{a(r^2+a^2)}{\sqrt{\Sigma}}\bigg[ - \partial_{\theta}\bigg(\dfrac{\sin ^2(\theta)}{\sqrt{\Sigma}}\bigg) + 
	\partial_{\theta}\bigg(\dfrac{\sin ^2(\theta)}{\sqrt{\Sigma}}\bigg)\bigg] = 0
	\ee
	\be
	\ce
	III = \dfrac{(r^2+a^2)^2}{\sqrt{\Sigma}}\bigg[ - \partial_{\theta}\bigg(\dfrac{\sin ^2(\theta)}{\sqrt{\Sigma}}\bigg) + 
	\partial_{\theta}\bigg(\dfrac{\sin ^2(\theta)}{\sqrt{\Sigma}}\bigg)\bigg] = 0
	\ee
	Thus, all terms with $\mu = 2$ and $(b)=2$ are zero. Finally, we end up with the last terms with $(b)=3$. \newline
	\underline{$\mu = 2$ and $(b)=3$}:
	\be
	\ce
	... &=\dfrac{1}{4}u_{(c)3}\gamma^{(c)}\gamma^5 \overbrace{\bigg[ \underbrace{\tilde{\epsilon}^{2 0 1 3}}_{=1}u_{(3)0}							\big(\partial_{\theta} u_{(3)1}\big) + \underbrace{\tilde{\epsilon}^{2 1 0 3 }}_{-1}u_{(3)1}\big(\partial_{\theta} 
		u_{(3)0}\big)\bigg]}^{=:I} \nonumber \\
	&+\dfrac{1}{4}u_{(c)1}\gamma^{(c)}\gamma^5 \overbrace{\bigg[ \underbrace{\tilde{\epsilon}^{2 3 0 1}}_{=1}u_{(3)3}							\big(\partial_{\theta} u_{(3)0}\big) + \underbrace{\tilde{\epsilon}^{2 0 3 1 }}_{-1}u_{(3)0}\big(\partial_{\theta} 
		u_{(3)3}\big)\bigg]}^{=:II} \nonumber \\
	&+\dfrac{1}{4}u_{(c)0}\gamma^{(c)}\gamma^5 \overbrace{\bigg[ \underbrace{\tilde{\epsilon}^{2 1 3 0}}_{=1}u_{(3)1}							\big(\partial_{\theta} u_{(3)3}\big) + \underbrace{\tilde{\epsilon}^{2 3 1 0 }}_{-1}u_{(3)3}\big(\partial_{\theta} 
		u_{(3)1}\big)\bigg]}^{=:III}
	\ee
	\be
	\ce
	I &= \dfrac{1}{4\sqrt{\Sigma}r_+^2}\bigg[ \big(\Delta + r_+^2\big) \partial_{\theta} \bigg(\dfrac{\Delta - 2 \Sigma - r_+^2}						{\sqrt{\Sigma}}\bigg) - \big( \Delta - 2 \Sigma - r_+^2 \big) \partial_{\theta} \bigg(\dfrac{\Delta + r_+^2}								{\sqrt{\Sigma}}\bigg)\bigg] \nonumber \\
	&= \dfrac{a^2}{\Sigma r_+^2}\big(\Delta - r_+^2\big)\cos(\theta)\sin(\theta)
	\ee
	\be
	\ce
	II &= \dfrac{a}{4 \sqrt{\Sigma}r_+^2}\big(\Delta - r_+^2 \big)^2\bigg[\partial_{\theta}\bigg( \dfrac{\sin ^2(\theta)}								{\sqrt{\Sigma}}\bigg) - \sin ^2(\theta)\big(\dfrac{1}{\sqrt{\Sigma}}\bigg) \bigg] \nonumber \\
	&= \dfrac{(\Delta - r_+^2)a}{2\Sigma r_+^2}\sin(\theta)\cos(\theta)
	\ee
	\be
	\ce
	III &= \dfrac{a}{4 \sqrt{\Sigma}r_+^2}\big(\Delta - r_+^2 \big)\bigg[ \sin ^2(\theta)\partial_{\theta}\bigg( \dfrac{\Delta - 2 					\Sigma -r_+^2}{\sqrt{\Sigma}}\bigg) - \big( \Delta - 2 \Sigma -r_+^2\big)\partial_{\theta}\bigg(\dfrac{\sin ^2(\theta)}						{\sqrt{\Sigma}} \bigg)\bigg] \nonumber \\
	&= \dfrac{a(\Delta - r_+^2 )}{4 \sqrt{\Sigma}r_+^2}\big( 2a^2 + 2r^2 +r_+^2 - \Delta \big) \cos(\theta) \sin(\theta)
	\ee
	Summing up over $(c)$. Note that the $\gamma^{(1)}\gamma^5$-term is trivially zero.\newline
	\underline{$\gamma^{(0)}\gamma^5$-Term}:
	\be
	\ce
	\dfrac{1}{4}\gamma^{(0)}\gamma^5\bigg[u_{(0)3}I + u_{(0)1}II + u_{(0)0}III\bigg] = \dfrac{a(\Delta - r_+^2)}								{4\sqrt{\Sigma}r_+}\cos(\theta)\sin(\theta)\gamma^{(0)}\gamma^5
	\ee
	\underline{$\gamma^{(1)}\gamma^5$-Term}:		
	\be
	\ce
	\dfrac{1}{4}\gamma^{(2)}\gamma^5\bigg[u_{(2)3}I + u_{(2)1}II + u_{(2)0}III\bigg] = 0
	\ee	
	\underline{$\gamma^{(3)}\gamma^5$-Term}:		
	\be
	\ce
	\dfrac{1}{4}\gamma^{(3)}\gamma^5\bigg[u_{(3)3}I + u_{(3)1}II + u_{(3)0}III\bigg] = 0
	\ee
	Dividing all terms with $\sqrt{|g|}$ we end up with the second term of \eqref{BTerm}:
	\be
	\ce
	-\dfrac{1}{4}\epsilon^{\mu \alpha \beta \delta}\eta^{(a)(b)}u_{(a)\alpha}\big(\partial_{\mu}u_{(b)\beta}\big)u_{(c)\delta}					\gamma^{(c)}\gamma^5 &= \dfrac{a(\Delta - r_+^2)}{4\Sigma^{3/2}r_+}\cos(\theta)\sin(\theta)\gamma^{(0)}\gamma^5  \nonumber \\
	&+ \dfrac{a(\Delta + r_+^2)}{4\Sigma^{3/2}r_+}\cos(\theta)\sin(\theta)\gamma^{(3)}\gamma^5 \nonumber \\
	&+ \dfrac{r a \sin ^2(\theta)}{2\sqrt{\Sigma}}\gamma^{(1)}\gamma^5
	\ee
	
	which together with \eqref{FirstTerm} concludes the proof. \hfill \qed
	
	\section{Proof of Eq. \ref{ZweiteGleichung}}
	\label{zweite}
	
	\V For the Dirac operator we need to compute a second term containing the partial derivatives
	\be
	\ce
	iG^{\mu}\dfrac{\partial}{\partial x^{\mu}} &= i \big( u^{\mu}_{(a)}\gamma^{(a)}\partial_{\mu}\big) \nonumber \\
	&= i\gamma^{(0)}\bigg[u^0_{(0)}\partial_{\tau} + u^1_{(0)}\partial_{r} + u^3_{(0)}\partial_{\phi}\bigg] \nonumber \\
	&+ i\gamma^{(1)}u^2_{(1)}\partial_{\theta}  + i\gamma^{(2)}\bigg[u^0_{(2)}\partial_{\tau} + u^3_{(2)}\partial_{\phi}\bigg]   						\nonumber \\
	&+ i\gamma^{(3)}\bigg[u^0_{(3)}\partial_{\tau} + u^1_{(3)}\partial_{r} + u^3_{(3)}\partial_{\phi}\bigg].
	\ee	
	Plugging in the Dirac matrices we can define differential operators of the form
	\be
	\ce
	D_1 &:= -\dfrac{i}{\sqrt{\Sigma}r_+}\bigg[\big( 2 r^2 + 2a^2 - \Delta )\partial_{\tau} + \Delta\partial_r + 2a								\partial_{\phi}\bigg]\nonumber \\
	D_0 &:= -\dfrac{ir_+}{\sqrt{\Sigma}}\bigg[ \partial_{\tau} - \partial_r \bigg] \nonumber \\
	A_1 &:= -\dfrac{i}{\sqrt{\Sigma}}\partial_{\theta} - \dfrac{1}{\sqrt{\Sigma}}\big(a\sin(\theta)\partial_{\tau} + \csc(\theta)							\partial_{\phi} \big) \nonumber \\
	A_2 &:= -\dfrac{i}{\sqrt{\Sigma}}\partial_{\theta} + \dfrac{1}{\sqrt{\Sigma}}\big(a\sin(\theta)\partial_{\tau} + \csc(\theta)						\partial_{\phi} \big),
	\ee
	which have the following positions in the $4\times 4$ matrix:
	\be
	\ce
	iG^{\mu}\dfrac{\partial}{\partial x^{\mu}} = 
	\begin{bmatrix}
		0& 0& D_1&A_1 \\
		0& 0& A_2&D_0 \\
		D_0 & -A_1 & 0 & 0\\
		-A_2 & D_1 & 0 & 0 \\
	\end{bmatrix}
	\label{eq:A2.3}
	\ee
	Combining \eqref{ErsteGleichung} and \eqref{eq:A2.3} we end up with the Dirac operator as in \eqref{ZweiteGleichung}. \hfill \qed

	\section{Proof of Eq. \ref{DritteGleichung}}
	\label{dritte}
	
	\noindent By choosing the transformation $D := \text{diag}(\bar{\delta}^{1/2}, (\bar{\delta}|\Delta|)^{1/2}, (\delta|\Delta|)^{1/2}, \delta^{1/2})$ with $\delta := (r+ia\cos(\theta))$ and $\bar{\delta}:= (r-ia\cos(\theta))$ we can express the Dirac-Equation in a much simpler form:
	\be
	\ce
	\big(G-m\big)D^{-1} = 
	\begin{bmatrix}
		-m \bar{\delta}^{-1/2} & 0 & \alpha_1(\delta|\Delta|)^{-1/2} & \beta_-\delta ^{-1/2} \\
		0 & -m (\bar{\delta}|\Delta|)^{-1/2} & \beta_+(\delta|\Delta|)^{-1/2} & \alpha_0\delta ^{-1/2} \\
		\bar{\alpha}_0\bar{\delta}^{-1/2} & \bar{\beta}_+(\bar{\delta}|\Delta|)^{-1/2} & -m(\delta|\Delta|)^{-1/2} & 0\\
		\bar{\beta}_-\bar{\delta}^{-1/2} & \bar{\alpha}_1(\bar{\delta}|\Delta|)^{-1/2} & 0 & -m(\delta|\Delta|)^{-1/2}\\
	\end{bmatrix}
	\ee
	\be
	\ce
	\alpha_1(\delta|\Delta|)^{-1/2} &= \overbrace{-\dfrac{i}{\sqrt{\Sigma}r_+}\bigg[ \big( 2 r^2 + 2a^2 - \Delta\big)\partial_{\tau} + 2a\partial_{\phi} \bigg] (\delta|\Delta|)^{-1/2}}^{=:C}\nonumber \\
	&-\dfrac{\Delta}{\sqrt{\Sigma}r_+}\bigg[ i \partial_r \bigg(\dfrac{1}{\sqrt{(r+ia\cos(\theta))|\Delta|}}\bigg) + i\dfrac{r-M}{\Delta\sqrt{(r+ia\cos(\theta))|\Delta|}}  \nonumber \\
	&+ \dfrac{i}{2\Sigma\sqrt{|\Delta|}}\bigg( \dfrac{r - ia\cos(\theta)}{\sqrt{r + ia\cos(\theta)}}\bigg) \bigg] \nonumber \\
	&= C - \dfrac{i \Delta}{\sqrt{\Sigma}r_+} (\delta|\Delta|)^{-1/2} \partial_r \nonumber \\
	&- \dfrac{i \sqrt{|\Delta|}}{2\sqrt{\Sigma}r_+} \bigg[-\dfrac{1}{(r+ia\cos(\theta))^{3/2}} + \dfrac{r - ia\cos(\theta)}{(r + ia\cos(\theta))^{3/2} (r- ia\cos(\theta))} \bigg]\nonumber \\
	&= C - (\delta|\Delta|)^{-1/2} \dfrac{i\Delta}{\sqrt{\Sigma}r_+}\partial_r = (\delta|\Delta|)^{-1/2} \alpha_1'
	\ee
	Likewise computation for $\bar{\alpha}_1$ leads to same result $ \bar{\alpha}_1(\bar{\delta}|\Delta|)^{-1/2} = (\bar{\delta}|\Delta|)^{-1/2} \bar{\alpha}_1' $ with $\bar{\alpha}_1' =\alpha_1' $:
	\be
	\ce
	\alpha_0\delta^{-1/2} &= \overbrace{-\dfrac{ir_+}{\sqrt{\Sigma}}\delta^{-1/2}\bigg[\partial_{\tau} - \partial_r \bigg]}^{=:C} + 			\dfrac{ir}{\sqrt{\Sigma}}\partial_r \delta^{-1/2} + \dfrac{ir_+}{2\Sigma^{3/2}}\bigg[r - ia\cos(\theta)\bigg]\nonumber \\
	&= C + \dfrac{ir_+}{2\sqrt{\Sigma}}\underbrace{\bigg(\dfrac{r-ia\cos(\theta)}{\Sigma} - \dfrac{1}{r+ia\cos(\theta)} \bigg)}_{=0}
	\nonumber \\
	&= \delta^{-1/2}\alpha_0'
	\ee
	Equivalent computation for $\bar{\alpha}_0$ leads to the same result $ \bar{\alpha}_0\bar{\delta}^{-1/2} = \bar{\delta}^{-1/2} \bar{\alpha}_0' $ with $\bar{\alpha}_0' =\alpha_0' $:
	\be
	\ce
	\beta_-\delta^{-1/2} &=   \overbrace{-\dfrac{1}{\sqrt{\Sigma}}\delta^{-1/2}\bigg[ i\partial_{\theta} +i \dfrac{\cot(\theta)}{2} + 								\big(a\sin(\theta)\partial_{\tau} + \csc(\theta)\partial_{\phi} \big)\bigg]}^{=:C} \nonumber \\
	&-\dfrac{1}{\sqrt{\Sigma}}\bigg[i\partial_{\theta}\dfrac{1}{\sqrt{(r+ia\cos(\theta))}} + + \dfrac{a												\sin(\theta)}{2\Sigma}\bigg(\dfrac{r - ia\sin(\theta)}{\sqrt{r + ia\sin(\theta)}}\bigg)\bigg]\nonumber \\
	&= C - \dfrac{a\sin(\theta)}{2(r + ia\sin(\theta))}\underbrace{\bigg(\dfrac{r-ia\cos(\theta)}{\Sigma} - 											\dfrac{1}{r+ia\cos(\theta)} \bigg)}_{=0} = \delta^{-1/2}\beta_-'
	\ee
	Identical computations for $\bar{\beta}_-$, $\beta_+$ and $\bar{\beta}_+$ lead to the same results $\bar{\beta}_-\bar{\delta}^{-1/2} =\bar{\delta}^{-1/2} \bar{\beta}_-' $ and $\bar{\beta}_+\bar{\delta}^{-1/2} =\bar{\delta}^{-1/2} \bar{\beta}_+'$ with $\bar{\beta}_-' = \beta_-'$ and $\bar{\beta}_+' = \beta_+'$ \footnote{The difference between $\beta_-'$ and $\beta_+'$ is only the sign in front of the last term.}. Thus, we end up with the full transformed matrix:
	\be
	\ce
	D\big(G-m\big)D^{-1} = 
	\begin{bmatrix}
		-m & 0 & |\Delta|^{-1/2}\bigg(\dfrac{\bar{\delta}}{\delta}\bigg)^{1/2}\alpha_1 ' & \bigg( \dfrac{\bar{\delta}}{\delta}						\bigg)^{1/2} \beta_- '\\
		0 & -m & \bigg(\dfrac{\bar{\delta}}{\delta}\bigg)^{1/2}\beta_+ ' & |\Delta|^{1/2}\bigg(\dfrac{\bar{\delta}}{\delta}\bigg)^{1/2}
		\alpha_0 ' \\
		|\Delta|^{1/2}\bigg(\dfrac{\delta}{\bar{\delta}}\bigg)^{1/2}\bar{\alpha}_0' & \bigg(\dfrac{\delta}{\bar{\delta}}\bigg)^{1/2} 				\bar{\beta}_+' & -m & 0 \\
		\bigg(\dfrac{\delta}{\bar{\delta}}\bigg)^{1/2}\bar{\beta}_-' & |\Delta|^{-1/2}\bigg(\dfrac{\delta}{\bar{\delta}}\bigg)^{1/2} 
		\bar{\alpha}_1 ' & 0 & -m\\
	\end{bmatrix}
	\ee
	Since following relations
	\be
	\ce
	\bigg(\dfrac{\delta}{\bar{\delta}}\bigg)^{1/2} \bar{\delta} = \bigg(\dfrac{\delta ^2}{\Sigma}\bigg)^{1/2}\bar{\delta} = 					\dfrac{1}{\Sigma}\delta\bar{\delta} = \dfrac{\Sigma}{\sqrt{\Sigma}} = \sqrt{\Sigma} = \bigg(\dfrac{\bar{\delta}}{\delta}					\bigg)^{1/2}
	\ee
	hold, we multiply the Dirac-Equation from the left with the matrix $\Gamma := -i\text{diag}(\delta, - \delta, -\bar{\delta}, \bar{\delta})$. In a last step we multiply the whole left hand side with $(-1)$. Finally, we end up with the result
	\be
	\ce
	\Gamma D \big(G-m\big) D^{-1} = 
	\begin{bmatrix}
		i\delta m & 0 & \sqrt{|\Delta|}^{-1}\tilde{\mathcal{D}_1} &\tilde{\mathcal{L}_+} \\
		0 & -i\delta m & -\tilde{\mathcal{L}_-} &\sqrt{|\Delta|}\tilde{\mathcal{D}_0}\\
		\sqrt{|\Delta|}\tilde{\mathcal{D}_0} & \tilde{\mathcal{L}_+} & -i\bar{\delta}m & 0 \\
		-\tilde{\mathcal{L}_-} & \sqrt{|\Delta|}^{-1}\tilde{\mathcal{D}_1} & 0 & i\bar{\delta}m \\
	\end{bmatrix}
	\ee
	with the differential operators
	\be
	\ce
	\tilde{\mathcal{D}_1} &= \dfrac{1}{r_+} \bigg[ \big(2r^2 + 2a^2 - \Delta\big) \partial_{\tau} + \Delta\partial_r 			+ 		2a\partial_{\phi}\bigg] \nonumber \\
	\tilde{\mathcal{D}_0} &=-r_+\bigg[\partial_{\tau} - \partial_r \bigg] \nonumber \\
	\tilde{\mathcal{L}_{\pm}} &= \bigg[ \partial_{\theta} + \dfrac{\cot(\theta)}{2} \bigg] \mp i \big(a\sin(\theta)\partial_{\tau} 			+ \csc(\theta)\partial_{\phi} \big).
	\ee
	This ends the proof. \hfill \qed


	\newpage
	\pagestyle{plain}



\end{document}